\documentclass[10pt]{amsart}
\usepackage[usenames, dvipsnames]{color}
\usepackage{cite}
\usepackage{mathrsfs}
\usepackage{graphicx}
\usepackage{amsmath}
\usepackage{amssymb}
\usepackage[linktoc=all]{hyperref}
\hypersetup{
  colorlinks,
  citecolor=black,
  filecolor=black,
  linkcolor=black,
  urlcolor=black
}
\usepackage{amsfonts}
\usepackage{amsthm}
\usepackage[shortlabels]{enumitem}
\newcounter{thm}
\newtheorem{theorem}[thm]{Theorem}
\newtheorem{lemma}[thm]{Lemma}
\newtheorem{corollary}[thm]{Corollary}
\newtheorem{proposition}[thm]{Proposition}

\theoremstyle{definition}
\newtheorem{definition}{Definition}

\newtheoremstyle{rmk}
{2pt}% Space above
{5pt}% Space below
{}% Body font
{}% Indent amount
{\itshape}%Theorem head font
{:}% Punctuation after theorem head
{.5em}% Space after theorem head 3
{}% Theorem head spec (can be left empty, meaning ‘normal’)
\newtheorem{remark}{Remark}

\usepackage{xypic}

\newcommand{\cond}[1]{&\text{if }#1}
\renewcommand{\tilde}{\widetilde}

\title{Higher-genus wall-crossing in Landau--Ginzburg theory}
\author{Yang Zhou}
\address{Department of Mathematics, Stanford University, 450 Serra Mall,
  Stanford, CA 94305, USA}
\email{yangzhou@stanford.edu}

\date{}
\begin{document}

\begin{abstract}
  For a Fermat quasi-homogeneous polynomial, we study the associated weighted
  Fan--Jarvis--Ruan--Witten theory with narrow insertions. We prove a wall-crossing
  formula in all genera via localization on a master space, which is constructed
  by introducing an additional tangent vector  to
    the moduli problem. This is a Landau--Ginzburg theory
  analogue of the higher-genus quasi-map wall-crossing formula proved by
  Ciocan-Fontanine and Kim. It generalizes the genus-$0$ result by Ross--Ruan and
  the genus-$1$ result by Guo--Ross.
\end{abstract}
\maketitle

\section{Introduction}
\subsection{Overview}
The goal of this paper is to prove the all-genera wall-crossing formula for
the  weighted Fan--Jarvis--Ruan--Witten (FJRW) invariants of a Fermat
polynomial with extended narrow  insertions.

Let $W$ be a Fermat type degree-$r$ quasi-homogeneous
polynomial of weights $(w_1,\cdots,w_s)\in \mathbb Z^s$:
\[
   W(X_1,\cdots,X_s)=X_1^{r/w_1} +\cdots + X_s^{r/w_s} .
\]
We assume that each
$r/w_\alpha\geq 2$ is an integer and $\gcd(r,w_1,\cdots,w_s)=1$.
The polynomial $W$ defines a smooth hypersurface in the weighted projective
space $\mathbb P(w_1,\cdots,w_s)$.
We define charges $q_\alpha=w_\alpha/r$ and set $q=\sum q_\alpha$.
When $q=1$, $X_W$ is a Calabi--Yau orbifold and  the Landau--Ginzburg/Calabi--Yau (LG/CY) correspondence relates the
Gromov--Witten theory of $X_W$ to the FJRW theory
of $(W,\langle J \rangle)$. These are the two phases
of the {\it gauged linear sigma model} (GLSM)  mathematically defined in
\cite{fan2015mathematical}.

There is a family of gauged linear sigma models interpolating between these two
theories \cite{witten1993phases, fan2015mathematical}, parametrized by a nonzero rational number $\epsilon$. As
$\epsilon$ varies, the change of the theories gives a
wall-and-chamber structure on the parameter space.

When $\epsilon>0$, it is the CY side. The theory associated to
sufficiently large  $\epsilon$ is the Gromov--Witten theory of $X_W$; the theory
associated to sufficiently small $\epsilon>0$ is the stable quotient invariants \cite{mustacta2007intermediate, toda2011moduli, marian2011moduli,ciocan2014stable, ciocan2010moduli}. In the
recent work \cite{ciocan2016quasimap}, Ciocan-Fontanine and Kim established
an explicit wall-crossing formula relating these two theories for complete
intersections in projective spaces.

When $\epsilon<0$, it is the LG side and we have weighted FJRW theories. We call the theory associated to
sufficiently negative $\epsilon$ the $\infty$-FJRW invariant theory; we call the theory associated to
$\epsilon$ sufficiently close to $0$ the $0$-FJRW invariant theory.
Our work is the LG analogue of the work by Ciocan-Fontanine and
Kim\cite{ciocan2016quasimap}. We will prove a similar wall-crossing formula relating
the $\infty$-FJRW theory and the
$0$-FJRW theory, for extended narrow insertions.

We expect that the LG side and the CY side are more
directly related near $\epsilon =0$. Using wall-crossing Ross--Ruan proved
the LG/CY correspondence in genus $0$ \cite{ross2014wall}. In genus $1$, Guo--Ross used the wall-crossing formula
to compute the FJRW invariants of the quintic $3$-fold explicitly
\cite{guo2016genus} and verified the genus-$1$ LG/CY correspondence
\cite{guo2017genus}.
Our result generalizes their wall-crossing formulas to all genera. We hope this will be
useful for establishing the all-genera LG/CY correspondence. The higher-genus
wall-crossing formula is also proved by Clader--Janda--Ruan \cite{clader2017higherLG} using different
methods.

\subsection{The statement of the results}

Let $\phi_1,\cdots,\phi_r$ be formal symbols and $H_W$ be the $\mathbb
Q$-vector space spanned by $\phi_1,\cdots,\phi_r$. Following \cite{ross2014wall}, we call
$H_W$ the extended narrow state space.

We fix non-negative integers $g,m,n$. For integers $a_i,b_j\in \{1,\cdots,r\}, c_i\geq
0$ indexed by $i=1,\cdots,m$ and $j=1,\cdots,n$, we
study the $0$-FJRW invariants
\begin{equation}
  \label{intro_inv}
  \langle
  \psi^{c_1}\phi_{a_1},\cdots,\psi^{c_m}\phi_{a_m}|\phi_{b_1},\cdots,\phi_{b_n}
  \rangle^{0}_{g,m|n}\in \mathbb Q,
\end{equation}
which are defined via intersection theory on the moduli space parametrizing
triples $(C,L,p)$, where $C$ is a Hassett-stable twisted curve with $m$
orbifold markings $x_1,\cdots,x_m$ of weight $1$ and $n$ non-orbifold markings
$y_1,\cdots,y_n$ of weight $\epsilon$, for sufficiently small $\epsilon>0$; $L$ is a representable line bundle with
appropriate monodromies at each orbifold marking $x_i$; and $p$ is a non-vanishing section
\begin{equation}
  \label{intro_spin}
  p\in H^0\Big(C,L^{-r}\otimes \omega_{C}\big(\sum_{i=1}^m x_i+\sum_{b_j\neq r} (1-b_j)y_j + \sum_{b_j= r} y_j\big)\Big).
\end{equation}
Here the $\psi$ are the cotangent-line classes at the
markings on the coarse curves.
Because the weights of $x_i$ are $1$, they are called heavy markings;  because
the weights of $y_j$ are arbitrarily small, they are called light markings.
The monodromy of $L$ at the orbifold point $x_i$ is determined by the state
$\phi_{a_i}$ in (\ref{intro_inv}).
The light markings $y_j$ are non-orbifold points and the $b_j$ in
(\ref{intro_spin}) play the role of the monodromy (Section \ref{other_description}).
We will give a self-contained definition of these invariants in Section
\ref{weighted_FJRW}. When $2g-2+m < 0$ or $2g-2+m=n=0$, the moduli spaces are empty and we define the
invariants (\ref{intro_inv}) to be zero.

For any formal power series
\[
  \mathbf u = u_0 + u_1 \psi + u_2\psi^2 + \cdots \quad\! \text{and}\! \quad  \mathbf t = \sum_{j=1}^r t_j\phi_j,\quad\text{where
  } u_i=\sum_{j=1}^r u_{ij}\phi_j,
\]
we use  (\ref{intro_inv}) and multi-linearity to define
\[
  \langle \mathbf u,\cdots,\mathbf u| \mathbf t,\cdots,\mathbf t
  \rangle_{g,m|n}^{0}\in \mathbb Q[\![\{u_{ij},t_i\}]\!].
\]
We form a generating function of $0$-FJRW invariants
\[
  \mathcal F^{0}_{g}(\mathbf u,\mathbf t) = \sum_{m,n\geq 0}
  \frac{1}{m!n!}\langle \mathbf u^m| \mathbf t^n
  \rangle_{g,m|n}^{0},
\]
where $\mathbf u^m$ means $\mathbf u,\cdots,\mathbf u$, repeated $m$-times.

We call $\phi_a$ a narrow state if $aq_\alpha\notin \mathbb Z$ for all
$\alpha=1,\cdots,s$. This means that for each $\alpha$, the line bundle
$L^{w_\alpha}$ has nontrivial monodromy at the marking where $\phi_a$
is ``inserted''. Otherwise we call $\phi_a$ a broad state.
The invariant (\ref{intro_inv}) vanishes unless all the $\phi_{a_i}$ are narrow (Lemma
\ref{lem_Ramond_vanishing}). This is referred to as the Vanishing Axiom in
\cite{polishchuk2004witten}.

The $\infty$-FJRW theory is a special case of the $0$-FJRW theory.
It is by definition the $0$-FJRW theory with no light markings.
Thus we have
\[
  \langle \mathbf u^m \rangle_{g,m}^{\infty} =
  \langle \mathbf u^m|\varnothing \rangle_{g,m|0}^{0}\in \mathbb Q[\![u_{ij}]\!]
\]
and
\[
  \mathcal  F^{\infty}_g(\mathbf u) = \mathcal F^{0}_{g}(\mathbf u,0).
\]

We now explain the analogy to the CY side.
A map to $\mathbb P^n$ consists of a line bundle and $n$ sections without common
zeros.
In the stable quotient theory, we allow some common zeros of those sections.
In the $0$-FJRW theory, assuming that $b_j=2 $ for all $j$, we can view the
light markings as ``common'' zeros, as follows. We look at the
image $\bar p$ of $p$ under the natural inclusion
\[
  H^0\Big(C,L^{-r}\otimes \omega_{C}\big(\sum_{i=1}^m x_i- \sum_{j=1}
  y_j \big)\Big) \longrightarrow H^0(C,L^{-r}\otimes \omega_{C}\big(\sum x_i)).
\]
Then $y_j$ become the zeros of $\bar p$. That is analogous to the stable quotient theory.
While in the $\infty$-FJRW theory, there are no light markings. Hence $\bar p=p$
is required to be non-vanishing.
Hence it is analogous to the Gromov--Witten theory.

We now state the numerical wall-crossing formula.
We first define an explicit $H_W$-valued series $\mu(\mathbf t,z)$ related to
the $I$-function.
As usual, for $x\in \mathbb Q$, we denote by $\left \lfloor x \right \rfloor$
the largest integer $\leq x$ and $\langle x \rangle=x-\left \lfloor x
\right \rfloor $ the fractional part of $x$.
For any $B_n=(b_1,\cdots,b_n)$, where $b_j\in \{1,\cdots,r\}$ for each $j$, we define $k_{B_n}$
to be the integer such that
\[
  1\leq k_{B_n}\leq r \quad \text{and} \quad  k_{B_n} -1 \equiv \sum_{j=1}^n(b_j-1) \mod r.
\]
For each $\alpha=1,\cdots,s$, we define
\[
  \ell_{\alpha,B_n} = \bigg \lfloor   \sum_{j=1}^n\langle
  q_\alpha(b_j-1)\rangle\bigg \rfloor \quad \text{and} \quad  k_{\alpha,B_n}  =
  q_\alpha +
  \big\langle  q_\alpha (k_{B_n}-1) \big\rangle.
\]
For each integer $n$, we write $[x]_n =x(x+1)\cdots(x+n-1)$. We  define
\[
  \mu_{B_n}(z) = \prod_{\alpha=1}^s \left[ k_{\alpha,B_n} \right]_{\ell_{\alpha,B_n}} z^{1-n+ \Sigma_{\alpha} \ell_{\alpha,B_n}}
\]
and
\[
  \mu(\mathbf t,z) =\sum_{n\geq 1}\sum_{B_n}
  \frac{t_{b_1}\cdots t_{b_n}}{n!}
  \mu_{B_n}(z) \phi_{k_{B_n}}.
\]
Let $\mu^+(\mathbf t,z)$ be the truncation of $\mu(\mathbf t,z)$ consisting of all non-negative
powers of $z$.
The big $\mathbb
I$-function defined in \cite{ross2014wall} is our $z\phi_1+\mu^+(\mathbf t,z)$
(cf. Remark \ref{rmk_narrow}).

In this paper we prove
\begin{theorem}
  For $g\geq 1$, we have
  \label{thm_higher_genus_potential}
  \[
    \mathcal F^{0}_{g}(\mathbf u,\mathbf t) = \mathcal
    F^\infty_g(\mathbf u+\mu^+(\mathbf t,-\psi)).
  \]
  For $g=0$, the same equation holds modulo linear terms in the variables $\{u_{ij}\}$.
\end{theorem}

This numerical wall-crossing formula can be generalized in two ways. We can allow
$\psi$-class insertions at light marking and we can compare the virtual fundamental classes in
the Chow groups. They are both included in Theorem \ref{thm_main}.

This theorem is proved independently by Clader, Janda and Ruan
\cite{clader2017higherLG}. Indeed their theorem includes the hybrid
model case, assuming the existence of at least one heavy marking.
Our proof and the proof given in \cite{clader2017higherLG} both use master spaces, and use localization
to derive explicit wall-crossing formulas. Apart from this, the two proofs took different directions.
The master space used in \cite{clader2017higherLG} is constructed by introducing a new line bundle
paired with additional sections, and the proof is by induction on genus.

The master space used in this proof is constructed by introducing an additional
tangent vector at one light marking; the fixed-point components
correspond to the correction terms in the wall-crossing formula, which allows us
to obtain the wall-crossing term directly. This is motivated by the variation of
GIT by Thaddeus~\cite{thaddeus1996geometric}.

We now consider the Calabi--Yau case $q=1$ and restrict ourselves to $\mathbf t = t\phi_2$.
Namely we consider those $B_n$ where all $b_j=2$ and we set $t_2=t$. It follows that
$\mu^+(t\phi_2,z)$ only has degree-$0$ and degree-$1$ terms in $z$.
Define power series  $I_0(t)=1+O(t)$ and $I_1(t)=O(t)$ via
\[
  \mu^+(t\phi_2,z) = (I_0(t)-1)z\phi_1 + I_1(t)\phi_2.
\]

We set $\mathbf u=0$ in Theorem \ref{thm_higher_genus_potential} and apply the
dilation equation  for
the $\infty$-FJRW invariants \cite{fan2013witten} to get
\begin{corollary}
  When $q=1$,
  \[
    I_0(t)^{2g-2}
    \sum_{d\geq 0}\frac{t^d}{d!} \langle \phi_2,\cdots,\phi_2 \rangle_{g,d}^{0}
    =
    \sum_{n\geq 0}\frac{1}{n!}\left(  \frac{I_1(t)}{I_0(t)}\right)^n
    \left\langle \phi_2,\cdots, \phi_2\right\rangle_{g,n}^\infty , \text{ }g>1;
  \]
and
  \begin{equation}
  \label{cor_genus_1}
    \sum_{d\geq 1}\frac{t^d}{d!} \langle \phi_2,\cdots,\phi_2 \rangle_{1,d}^{0}
    =
    -\log(I_0(t))\langle \psi\phi_1 \rangle_{1,1}^\infty + \sum_{n\geq
1}\frac{1}{n!}\left( \frac{I_1(t)}{I_0(t)}\right)^n \left\langle \phi_2,\cdots,
\phi_2\right\rangle_{1,n}^\infty.
  \end{equation}
\end{corollary}
Formula (\ref{cor_genus_1}) recovers  Theorem 1.1 in \cite{guo2016genus}. Our
definition is slighted different from theirs. This is discussed in Section \ref{treat_light_markings}.
We can compute the genus-$0$ invariants and match the result of \cite{ross2014wall}, this is explained in
Section \ref{genus_0_case}.

\subsection{Plan of the paper}
In Section \ref{weighted_FJRW}, we briefly recall the definition of weighted
FJRW invariants. In Section \ref{Hassett_master},
we focus on the underlying coarse curves and construct a master space that
contains Hassett's moduli of weighted pointed curves of various weights.
The main result is the
properness of the master space.
In Section \ref{r_spin_master}, we define the $r$-spin master space.
% of stable $r$-spin curves with mixed weighted markings.
We introduce a $\mathbb C^*$-action on the master space and study the
fixed-point components.
In Section \ref{phi_fields_master}, we construct an equivariant
perfect obstruction theory with an equivariant cosection. The localization
formula will give us a relation among different $0$-FJRW theories. In
Section \ref{wall_crossing_formulas}, we collect and package the relations from
the localization formula and prove our main theorems, including Theorem
\ref{thm_higher_genus_potential}. In Section \ref{genus_0_case}, we study the
genus-$0$ invariants and match the main theorem of \cite{ross2014wall}.

\subsection{Conventions}
In this paper we will work over the field of complex numbers. All scheme are
assumed to be locally noetherian over $\mathbb C$.  The genus of a curve
is always the arithmetic genus.
When $C$ is an orbifold curve and $x$ is
a marking of orbifold index $r$ (i.e. the automorphism group of $x$ is cyclic of
order $r$), the line bundle $\mathcal O_C(x)$ has degree $1/r$, as opposed to $1$.

For $x\in \mathbb Q$, we denote by $\left \lfloor x \right \rfloor$ the greatest integer no larger that
$x$ and $\langle x \rangle=x-\left \lfloor x \right \rfloor $ the fractional
part of $x$. For an integer $n\geq 0$, we abbreviate $x(x+1)\cdots(x+n-1)$ to $[x]_n$.

\subsection{Acknowledgments}
I am grateful to my advisor Jun Li for his inspiration and careful
guidance. I would like to thank the organizers of following conferences: the 2016 Chengdu
International Conference on Gromov--Witten theory, the 2016 Workshop on Global Mirror
Symmetry at Chern Institute and the 2017 FRG workshop ``Crossing the Walls in
Enumerative Geometry'' at Columbia. These conference exposed me to the most recent process
in the subject as well as many inspiring ideas. The talks by Ionu\c t Ciocan-Fontanine,
Shuai Guo, Felix Janda and Dustin Ross were especially helpful. I am also
grateful for helpful discussions with Huai-Liang Chang, Honglu Fan and Ming
Zhang.

\section{The weighted FJRW invariants}
\label{weighted_FJRW}

In this section, we give a self-contained description of the weighted FJRW
theory within the scope of this paper. This is a special case of the gauged
linear sigma models
defined in \cite{fan2015mathematical}. The construction of the virtual
fundamental class follows \cite{chang2015witten}.
\subsection{The moduli spaces of stable $r$-spin curves with weighted markings}

We fix non-negative integers $g,m,n$ such that $2g-2+m\geq 0$ and
$(2g-2+m,n)\neq (0,0)$. We will define the moduli spaces of genus-$g$
stable $r$-spin curves with $m$ heavy markings and $n$ light markings.
They are indexed by the discrete data
\[
  \gamma =
  \big(\frac{a_1}{r},\dots,\frac{a_m}{r}|\frac{b_1}{r},\dots,\frac{b_n}{r}\big),
  \quad a_i,b_j\in \{1,\cdots,r\}.
\]

Let $S$ be any scheme.
\begin{definition}
  \label{weighted_r_spin}
 An $S$-family of pre-stable  $r$-spin curves of genus $g$ with $\gamma$-weighted markings
is the datum
\[
  \xi=(C,\pi,x_1,\cdots,x_m;y_1,\dots ,y_n,L,p),
\]
where
\begin{itemize}
\item (curve) $\pi:C\to S$ is an $S$-family of prestable twisted curves of genus
  $g$ with balanced nodes and orbifold markings $x_1,\cdots,x_m$;
\item (heavy markings) $x_1,\dots,x_m$ are disjoint from each other;
\item (light markings) $y_1,\dots,y_n$ are not necessarily disjoint markings
  contained in $C^{\mathrm{sm}}\backslash \{x_1,\cdots,x_m\}$;
\item (line bundle) $L$ is a representable line bundle on $C$;
\item ($p$-field) $p\in H^0(C,P)$ is
  non-vanishing, where
  \[
    P = L^{-r}\otimes \omega_{C/S}\big(\sum_{i=1}^m x_i+\sum_{b_j\neq r} (1-b_j)y_j + \sum_{b_j= r} y_j\big).
  \]
\end{itemize}
\end{definition}
We will abbreviate $x_1,\cdots,x_m$ to $\mathbf x$ and abbreviate
$y_1,\cdots,y_n$ to $\mathbf y$. Thus $\xi = (C,\pi,\mathbf x;\mathbf y,L,p)$.
\begin{remark}
  Following\cite{abramovich2002compactifying} (cf.\cite{chiodo2008stable}),
  a family of
twisted curves $\pi:C\to S$  with  balanced nodes and orbifold markings $x_1,\cdots,x_m$
is a proper flat morphism of relative dimension $1$ from a
Deligne--Mumford stack $C$ to $S$, with closed substacks $x_i\subset C$, such that
\begin{enumerate}
\item the fibres are connected $1$-dimensional with at worst nodal
  singularities, the coarse moduli of $C$ is a nodal curve of genus $g$ over $S$ .
\item a local model of a node is $[U/\pmb{\mu}_{r^\prime}]\to T$ with
  $T=\mathrm{Spec~}A$, $U = \mathrm{Spec~}A[z,w]/(zw-t)$ for some $t\in A$, and
  the action is given by $(z,w)\mapsto (\zeta_{r^\prime}z,\zeta_{r^\prime}^{-1}w)$.
\item each $x_i$ is a closed substack of the relatively smooth locus
  $C^{\mathrm{sm}}\subset C$ and $\pi|_{x_i}$  is an \'etale gerbe banded by $\pmb\mu_{r^\prime}$.
\item $C^{\mathrm{sm}}\backslash \{x_1,\cdots,x_m\}$ is an algebraic space.
\end{enumerate}
The line bundle $L$ being representable means that near each orbifold point $x$
of $C$, the
automorphism group of $x$ acts faithfully on the fibre $L|_{x}$ of $L$ at $x$.
That  $L$ has monodromy $\frac{a_i}{r}$ at the orbifold marking $x_i$ means the generator
of $\pmb{\mu}_{r^\prime}$ acts
on $L|_{x_i}$ as multiplication by $\exp(\frac{2a_i\pi}{r}\sqrt{-1})$.
Representability of $L$
implies that $r^\prime = r/\gcd(a_i,r)$ is determined by $a_i$ and $r$. To
test the choice of the generator, we agree that
near $x_i$, $L$ is isomorphic to $L^\prime (\frac{a_ir^{\prime}}{r}x_i)$, where $L^\prime$
is the pullback of some line bundle on the coarse moduli space of $C$.

For degree reasons, such a prestable $r$-spin curve exists if and only if
\begin{equation}
  \label{eq:selection}
   2g-2 + \sum_{i=1}^m (1-a_i) + \sum_{j=1}^n (1-b_j) \equiv 0 \mod r.
\end{equation}
We assume that (\ref{eq:selection}) holds true throughout the paper.
\end{remark}
\begin{definition}
  \label{defn_pos}
  The datum $\xi$ is said to be stable if the $\mathbb Q$-line bundle
$\omega_{C/S}(\sum x_i + \epsilon\sum y_j)$ is relatively ample for all $\epsilon \in \mathbb
Q_{>0} $.
\end{definition}

We  denote the moduli space parametrizing such stable $\xi$ by $\overline
M^{1/r}_{g,\gamma}$. This is a smooth Deligne--Mumford stack of
dimension $3g-3+m+n$.

\subsection{The obstruction theory and the virtual fundamental class}
\label{obstruction_theory_single}
Let $\overline M^{1/r,\varphi}_{g,\gamma}$ be the stack of $S$-families of
genus-$g$ stable $r$-spin curves with $\gamma$-weighted markings and $\varphi$-fields:
\[(C,\pi,x_1,\cdots,x_m;y_1,\cdots,y_n,L,p,\varphi_1,\cdots,\varphi_s),\]
where
\[
  (C,\pi,\mathbf x;\mathbf y,L,p)\in \overline M^{1/r}_{g,\gamma}(S)
\]
and
\[
  \varphi_{\alpha} \in H^0\big(C,L^{w_\alpha}(D_{\alpha})\big), \quad
  D_{\alpha} = - \sum_{q_\alpha a_i \in \mathbb
     Z}x_i+\sum_{b_j\neq r}\left \lfloor
     q_\alpha (b_j-1) \right \rfloor y_j -\sum_{b_j= r}y_j.
\]
The $\varphi_{\alpha}$ are called $\varphi$-fields.

Let $\pi:{\mathcal  C}\to \overline M^{1/r,\varphi}_{g,\gamma}$ be the universal
curve, $\omega_{\pi}$ be the relative dualizing sheaf and
$\mathcal  L$ be the universal line bundle.
We use $x_i,y_j$ and $D_\alpha$ to denote the
divisors on $\mathcal  C$ as in the definition of $\overline M^{1/r,\varphi}_{g,\gamma}$.

Let $\tau:\overline M^{1/r,\varphi}_{g,\gamma}
\to \overline M^{1/r}_{g,\gamma}$ be the forgetful map and $\mathbb L_\tau$ be
its relative cotangent complex.
Then $\tau$ admits
a relative perfect obstruction theory
\[
  \left(  \mathbb L_{\tau}\right)^\vee \longrightarrow \bigoplus_{\alpha=1}^s R\pi_*  \mathcal L^{w_\alpha}(D_\alpha).
\]
The obstruction sheaf $\bigoplus_\alpha R^1\pi_* \mathcal
L^{w_\alpha}(D_\alpha)$ has a cosection
\[
  \sigma: \bigoplus_{\alpha=1}^s R^1\pi_* \mathcal  L^{w_\alpha}(D_\alpha)
  \longrightarrow  R^1\pi_*\omega_{\pi} \cong \mathcal  O_{\overline
    M^{1/r,\varphi}_{g,\gamma}}
\]
defined by
\[
  (\dot \varphi_1,\cdots,\dot\varphi_s) \longmapsto p\sum_{i=\alpha}^s
  \dot\varphi_\alpha \partial_\alpha W(\varphi_1,\cdots,\varphi_s).
\]
We now explain why $\sigma$ is well-defined.
Note that for each $\alpha$ and each $x_i$ such that $q_\alpha a_i \not\in \mathbb Z$,
$\mathcal  L^{w_\alpha}$ has nontrivial monodromy at $x_i$. Hence the
section $\varphi_\alpha$ must vanish along $x_i$ and thus is a section of
\[
  \pi_*\big(\mathcal  L^{w_\alpha}(
  - \sum_{i=1}^mx_i + \sum_{b_j\neq r}\left \lfloor
    q_\alpha (b_j-1) \right \rfloor y_j-\sum_{b_j= r}y_j)\big).
\]
Hence $p \dot\varphi_\alpha \partial_\alpha W(\varphi_1,\cdots,\varphi_s)$ is a
section of $ R^1\pi_* \omega_{\pi}\big(\!\!-\!\Delta \big)$
where
\[
\Delta = \sum_{i=1}^m (\frac{r}{w_\alpha}-2)x_i + \sum_{q_\alpha a_i\in
    \mathbb Z} x_i + \sum_{b_j\neq r}(b_j-1 - \frac{1}{q_\alpha}\left \lfloor
    q_\alpha(b_j-1) \right \rfloor ) y_j + \sum_{b_j = r}(\frac{r}{w_\alpha}-1) y_j.
\]
Since we have assumed that $r/w_\alpha\geq 2$, we see that $\Delta$ is effective
and hence $R^1\pi_*\omega_{\pi}(-\Delta)$
 naturally maps to $R^1\pi_*\omega_{\pi}$.

The degeneracy loci $D(\sigma)$ of $\sigma$ is defined to be the locus where
$\sigma$ is not surjective. As in \cite{chang2015witten}, Serre duality implies that $D(\sigma)$ is equal to $\overline
M^{1/r}_{g,\gamma}$, viewed as a closed subset of $\overline
M^{1/r,\varphi}_{g,\gamma}$  where all the $\varphi$-fields
are identically zero.
The cosection lifts to the absolute obstruction theory  and defines a cosection localized
virtual fundamental class \cite{behrend1997gromov,li1998virtual,chang2015witten}
$[\overline M^{1/r,\varphi}_{g,\gamma}]_{\mathrm{loc}}^{\mathrm{vir}}\in
A_*(\overline M^{1/r}_{g,\gamma})$
in degree
\[
  (3-s+2q)(g-1) + m + n - \sum_{\alpha=1}^s\Big(
  \sum_{i=1}^m \langle q_\alpha (a_i-1) \rangle + \sum_{j=1}^n \langle
  q_\alpha (b_j-1) \rangle\Big).
\]
\subsection{Ramond vanishing}
We say that the heavy marking
$x_i$ is narrow if $q_\alpha a_i\not\in \mathbb Z$  for all $\alpha$. Otherwise
we say $x_i$ is broad. Broad heavy markings will naturally appear in our construction
even if we have started with narrow heavy markings only.

However, note that as the $\phi_\alpha$ are required to vanish along the broad heavy
markings, they are indeed ``fake'' broad markings, and we have the following
vanishing result from \cite{polishchuk2004witten}:

\begin{lemma}[\!\!\cite{polishchuk2004witten}]
  \label{lem_Ramond_vanishing}
  The class  $[\overline
  M^{1/r,\varphi}_{g,\gamma}]_{\mathrm{loc}}^{\mathrm{vir}}$ vanishes unless all
  heavy markings are narrow.
\end{lemma}
\begin{proof}
  If there are no light markings, this follows from Theorem 2.1 of
  \cite{polishchuk2004witten}. The equivalence between the cosection
  construction and Polishchuk's construction is established in
  \cite{chang2015witten}.
  We are not able not find a direct and explicit reference in the presence of light
  markings. Instead, if $g=0$,  the lemma follows from the proof of Lemma 1.8 in
  \cite{ross2014wall};
  if $g\geq 1$, we appeal to the wall-crossing in Theorem \ref{thm_main} and the lemma follows
  from the case when there are no light markings.
\end{proof}

\begin{remark}
The main theorems in this paper does not depend on  Lemma
  \ref{lem_Ramond_vanishing}.
 Indeed Lemma \ref{lem_Ramond_vanishing} only implies that many terms in the
 wall-crossing formulas turn out to be zero. Hence there is no circular reasoning in the proof above.
\end{remark}
\subsection{The extended narrow state space and weighted FJRW invariants}
\label{subsec_invariants}
We have introduced the $\mathbb Q$-vector space $H_W$ whose basis consists of the
formal symbols
$\phi_1,\cdots,\phi_{r}$. Let $\gamma$ be as before and
$\psi_{x_i},\psi_{y_j}$ be the cotangent-line classes on the
coarse curves. We define the $0$-FJRW invariants with descendants to be
\begin{equation}
  \label{defn_invariants}
  \begin{aligned}
   \langle  \psi^{c_1}\phi_{a_1},\cdots,\psi^{c_m}\phi_{a_m}|&
  \psi^{d_1}\phi_{b_1},\cdots,\psi^{d_n}\phi_{b_n}\rangle^{0}_{g,\gamma}                       \\
  :=& \epsilon_\gamma\cdot\int_{[\overline
    M^{1/r,\varphi}_{g,\gamma}]_{\mathrm{loc}}^{\mathrm{vir}}}
  \psi_{x_1}^{c_1}\cdots\psi_{x_m}^{c_m}
  \psi_{y_1}^{d_1}\cdots\psi_{y_n}^{d_n}.
  \end{aligned}
\end{equation}
where the constant
\[
  \epsilon_\gamma  = \frac{1}{r^{g-1}} (-1)^{(2q-s)(g-1)-\sum_{\alpha}(\sum_i
    \langle q_\alpha(a_i-1) \rangle + \sum_j \langle q_\alpha(b_j-1) \rangle)}
\]
is introduced to be consistent with the original definition of Fan--Jarvis--Ruan.
The sign is just $(-1)^{\Sigma_\alpha {\chi (R\pi_*\mathcal
    L^{w_\alpha}(D_\alpha))}}$. The $\infty$-FJRW invariants are defined to be
\[
   \langle  \psi^{c_1}\phi_{a_1},\cdots,\psi^{c_m}\phi_{a_m}\rangle^{\infty}_{g,\gamma}
  := \langle  \psi^{c_1}\phi_{a_1},\cdots,\psi^{c_m}\phi_{a_m}|\varnothing\rangle^{0}_{g,\gamma}.
\]
Both of them are referred to as weighted FJRW invariants.

\subsection{The twisted theory}
We can also consider the twisted theory.
Let \[(\mathcal  C,\pi,\mathbf x;\mathbf y,\mathcal  L,p) \] be the
universal family over $\overline{M}_{g,\gamma}^{1/r}$.
As before let $D_\alpha$ be the
divisor on $\mathcal  C$ defined by
\[
 D_{\alpha} = - \sum_{q_\alpha a_i \in \mathbb
     Z}x_i+\sum_{b_j\neq r}\left \lfloor
     q_\alpha (b_j-1) \right \rfloor y_j -\sum_{b_j= r}y_j.
\]
Let $\mathbb S=(\mathbb
C^*)^s$ be the torus that acts trivially on $\overline{M}_{g,\gamma}^{1/r}$ and
acts on the fibres of $\bigoplus_{\alpha=1}^s
\mathcal  L^{w_\alpha}(D_\alpha)$ by weights $(w_1,\cdots,w_s)$. We denote the
equivariant parameters by $\lambda_1,\cdots,\lambda_s$.
We define the extended narrow equivariant state space
\[
  H_W^{\lambda} = H_W\otimes \mathbb Q(\lambda_1,\dots,\lambda_s),
\]
and the equivariant virtual fundamental class
\[
  \big[ \overline{M}^{1/r}_{g,\gamma} \big]^{\mathrm{vir}}_{\mathbb S} = \big[
  \overline{M}^{1/r}_{g,\gamma}\big] \cap c_{\mathrm{top}}^{\mathbb
    S}\Big(\bigoplus_{\alpha=1}^s R\pi_*
\mathcal  L^{w_\alpha}(D_\alpha)[-1] \Big).
\]
The twisted weighted FJRW invariants are defined in the same way as in the
untwisted case in Section \ref{subsec_invariants}, with
$\big[\overline M^{1/r,\varphi}_{g,\gamma}\big]_{\mathrm{loc}}^{\mathrm{vir}}$ replaced by
$\big[ \overline{M}^{1/r}_{g,\gamma} \big]^{\mathrm{vir}}_{\mathbb S}$. We will focus on the
untwisted case. The results of this paper also work in the twisted case if we
modify one definition (Remark \ref{rmk_twisted}).
\subsection{Orbifold markings v.s. non-orbifold markings}
\label{other_description}

We have treated heavy markings as orbifold markings and light markings as
non-orbifold markings. Actually orbifold markings and non-orbifold markings are
equivalent in the following sense.
Suppose $\rho:C\to |C|$ is the partial coarse moduli forgetting the orbifold
structure only at $x_1,\cdots,x_m$ but not at the nodes. Let $\bar x_i$, $\bar y_j$ be the images of
$x_i,y_j$. Then we have a natural isomorphism
\begin{equation}
  \label{no_orbi_varphi}
  \begin{aligned}
  \rho_*\big(&L^{w_\alpha}(D_\alpha)\big) \cong   \\
  &(\rho_* L)^{w_\alpha} \otimes {\mathcal  O_{|C|}} \big(
  \sum_{a_i\neq r}\left \lfloor
    q_\alpha (a_i-1) \right \rfloor \bar x_i -\sum_{a_i= r} \bar x_i
  +\sum_{b_j\neq r}\left \lfloor
    q_\alpha (b_j-1) \right \rfloor \bar y_j -\sum_{b_j= r}\bar y_j
  \big).
  \end{aligned}
\end{equation}
Moreover for
\[
  P = L^{-r}\otimes \omega_{C}\big(\sum_{i=1}^m x_i +\sum_{b_j\neq r} (1-b_j)y_j+\sum_{b_j=r}y_j\big),
\]
we have
\begin{equation}
  \label{no_orbi_P}
  \rho_*P =(\rho_*L)^{-r}\otimes\omega_{|C|}\big(\sum_{a_i\neq r}(1-a_i)\bar x_i +\sum_{a_i=r}\bar x_i
  + \sum_{b_j\neq r}(1-b_j)\bar y_j+\sum_{b_j=r}\bar y_j\big).
\end{equation}
By  Theorem 4.2.1 of \cite{abramovich2008gromov}, there is a unique way to add
the stack structure at $x_i$.
Hence in Definition \ref{weighted_r_spin}, if we require that the curve has no
orbifold structure at each $x_i$, and replace the definition of $P$
by the right hand side of (\ref{no_orbi_P}) with $L$ in place of $\rho_*L$, we
get the same moduli space. Moreover, since $\rho_*$ is an exact functor, if we define the
$\varphi$-fields as sections of the right hand side of (\ref{no_orbi_varphi}), Section
\ref{obstruction_theory_single} works verbatim with the new definition and
defines the same virtual fundamental class under the natural identification
of the moduli spaces. In this sense we are treating heavy markings and light markings on
an equal footing.

We can also replace only a subset of orbifold heavy
markings by non-orbifold heavy markings. This will be useful when we
``transform'' a light marking into a heavy marking in Section \ref{phi_fields_master}.
\begin{remark}
  We insist that the heavy markings are orbifold markings because when we
  ``split'' a node, we get a pair of orbifold heavy markings; we insist that the
  light markings are non-orbifold markings because we do not want to consider
  two colliding orbifold markings.
\end{remark}

\subsection{Various treatments of the light markings}
\label{treat_light_markings}
There are at least two different ways to treat the light markings. In
\cite{ross2014wall} and in this paper, the light markings are ordered. While in
\cite{fan2015mathematical, guo2016genus}, the light markings are unordered,
meaning that we only remember the divisor $y_1+\cdots+y_n$.
These two treatments are equivalent, differing by a factor of $n!$.
The argument is simple and we omit it here.

\section{The moduli of stable curves with mixed weighted markings}
\label{Hassett_master}
In \cite{hassett2003moduli}, Hassett defined the moduli of stable curves with
weighted markings. In this paper, we consider the following special case:  moduli space $\overline M_{g,m|n}$ of stable curves
$\pi:C\to S$ with $m$ weight-$1$ markings $x_1,\cdots,x_m$ and $n$ weight-$\epsilon$ markings
$y_1,\cdots,y_n$, where $\epsilon>0$ is sufficiently small.  The stability
condition for $\overline M_{g,m|n}$ is
\begin{enumerate}
\item $\pi:C\to S$ is a family of nodal curves and all the markings are contained in the relative smooth locus of $C$;
\item each $x_i$ does not intersect any other markings;
\item the $\mathbb Q$-line bundle $\omega_{C/S}(\sum x_i + \epsilon \sum y_j)$
  is relatively ample for all $\epsilon\in \mathbb Q_{>0}$.
\end{enumerate}
We call $x_i$ the heavy markings and  $y_j$ the light markings.

In this section, we will construct the ``master'' moduli space $\widetilde{M}_{g,m|n}$ of
genus-$g$ stable curves with {\it mixed} $(m,n)$-weighted markings
$x_1,\cdots,x_m;y_1,\cdots,y_n$. In this moduli, the $x_i$ behave like heavy
markings; $y_2,\cdots,y_n$ behave like light markings; $y_1$ is a ``mixed
weighted'' marking. The stack $\widetilde{M}_{g,m|n}$ contains both
$\overline{M}_{g,m|n}$ and $\overline{M}_{g,m+1|n-1}$ as closed substacks. We
do not consider $r$-spin structures in this section. Thus we only consider
non-orbifold curves here.

\subsection{The construction of $\tilde M_{g,m|n}$}
\label{subsec_curves}
We fix non-negative integers $g,m,n$ such that $2g-1+m\geq 0$, $n\geq 1$ and
$(2g-1+m,n-1)\neq (0,0)$.
Let $S$ be any scheme.
\begin{definition}
  \label{defn_M_tilde}
  An $S$-family of genus-$g$ stable curves with {\it mixed} $(m,n)$-weighted markings
  consists of $(C,\pi,x_1,\cdots,x_m;y_1,\cdots,y_n,N,v_1,v_2)$, where
  \begin{enumerate}
  \item $\pi:C\to S$ is a flat proper family of connected genus-$g$ nodal
    curves;
  \item
    $x_1,\cdots,x_m,y_1,\dots,y_n$ are markings
    contained in the relative smooth locus $C^{\mathrm{sm}}\subset C$;
  \item $N$ is a line bundle on $S$;
  \item $v_1\in H^0(S,T_{y_1}\otimes_{\mathcal  O_S} N)$ and  $v_2 \in H^0(S,N)$, where $T_{y_1} =
    \omega_{C/S}^\vee |_{y_1}$ .
  \end{enumerate}
  such that
  \begin{enumerate}
  \item each $x_i$ is disjoint from all other markings;
  \item $v_1$ and $v_2$ do not have any common zero on $S$;
  \item the $\mathbb Q$-line bundle $\omega_{C/S}(\sum_{i=1}^m x_i + y_1+
    \epsilon \sum_{j=2}^n y_j)$ is relatively ample for all $\epsilon\in
    \mathbb Q_{>0}$;
  \item when $v_1=0$, $y_1$ does not intersect other light markings
    $y_j,j=2,\cdots,n$;
  \item when $v_2=0$, the $\mathbb Q$-line bundle $\omega_{C/S}(\sum_{i=1}^m x_i+
    \epsilon \sum_{j=1}^n y_j)$ is relatively ample for all $\epsilon\in
    \mathbb Q_{>0}$;
  \end{enumerate}
\end{definition}

Let
\[
  \xi^\prime=(C^\prime,\pi^\prime,\mathbf x^\prime;\mathbf y^\prime,N^\prime,v^\prime_1,v^\prime_2)
  \quad \text{and} \quad  \xi = (C,\pi,\mathbf x;\mathbf y,N,v_1,v_2)
\]
be  two families of genus-$g$ stable curves with
mixed $(m,n)$-weighted markings over $S^\prime$ and $S$, respectively.
An arrow $\xi^\prime\to \xi$ consists of fibred diagram
\[
  \xymatrix{
    C^\prime \ar[r]^f\ar[d]_{\pi^\prime} & C \ar[d]^\pi\\
    S^\prime \ar[r]^{g} & S}
\]
and an isomorphism $\eta:N^\prime\to g^*N$ of line bundles
such that $f$ pulls back the markings to the corresponding markings,
$\eta(v_2^\prime)=g^*v_2$ and $(df_{y^\prime_1}\otimes \eta) (v_1^\prime) = g^*v_1$.

This defines the category $\widetilde{M}_{g,m|n}$ of genus-$g$ stable curves
with mixed $(m,n)$-weighted markings. It is fibred in groupoids over the category of schemes .
\begin{theorem}
  \label{representability}
  The category $\widetilde{M}_{g,m|n}$ is a
  smooth Deligne--Mumford stack of finite type over $\mathbb C$, of dimension $3g-2+m+n$.
\end{theorem}
\begin{proof}
  Let $\mathfrak M$ be the Artin stack of genus-$g$ nodal curves $C$
  with $m+n$ not necessarily distinct markings $x_1,\cdots,x_m,y_1,\cdots,y_n$
  in the smooth locus of $C$,
  such that  $C$ has at
  most $2g-2+m+n$ irreducible components. This is a finite type smooth Artin
  stack of dimension $3g-3+m+n$. Let $T_{y_1}$ be the line
  bundle on $\mathfrak M$ formed by the tangent spaces to the curves at $y_1$. Let $\mathbb  P$
  be the projective bundle $\mathbb P_{\mathfrak M}(T_{y_1}\oplus \mathcal
  O_{\mathfrak M})$ over $\mathfrak M$. Then $\widetilde{M}_{g,m|n}$ is represented by an open substack
  of $\mathbb  P$, hence represented by an Artin stack of finite type.
  It is easy to see that each closed point of
  $\widetilde{M}_{g,m|n}$ has a finite automorphism group. Hence
  $\widetilde{M}_{g,m|n}$ is a Deligne--Mumford stack \cite{olsson2016algebraic}.
  \end{proof}
\begin{remark}
  \label{rmk_effect_v_2}
  If $v_2$ is nowhere vanishing, it gives an isomorphism $N \cong
  \mathcal O_S$ sending $v_2$ to $1$. Hence the $(N,v_1,v_2)$ part of $\xi$ is
  equivalent to $v_1/v_2\in H^0(S,T_{y_1})$. When $v_2=0$, $v_1$ is
  non-vanishing and gives an isomorphism $N\cong T_{y_1}^\vee$. Thus
  at every closed point $s\in S$, we can view $(N,v_1,v_2)$ as a point of
  $T_{y_1}C_s\cup \{\infty\}$.
\end{remark}
\begin{remark}
  The universal $v_1$ and $v_2$ are sections of certain line bundles over  $\widetilde{M}_{g,m|n}$.
  The vanishing locus of $v_1$ is isomorphic to $\overline
  M_{g,m+1|n-1}$, where $y_1$ is a heavy marking; the vanishing locus of $v_2$ is isomorphic to $\overline
  M_{g,m|n}$, where $y_1$ is a light marking.

\end{remark}
\subsection{The properness of $\widetilde{M}_{g,m|n}$}
\begin{theorem}
  The stack $\widetilde{M}_{g,m|n}$ is proper.
\end{theorem}
\begin{proof}
  % Setup
  We prove the properness by the valuative criterion.
  Let $R$ be a Henselian DVR with residue field $\mathbb C$.
  Let $B=\mathrm {Spec\!~}R$, $b\in B$ be the
  closed point and $B^* = B\backslash \{b\}$ be the generic point.
  Suppose $\xi^* = (C^*,\pi^*,\mathbf x^*;\mathbf y^*,N^*,v^*_1,v^*_2)$ is in
  $\widetilde M_{g,m|n}(B^*)$.  We want to show that possibly after a finite
  base-change, we can extend $\xi^*$ to a family
  $\xi = (C,\pi,\mathbf x;\mathbf y,N,v_1,v_2)$ over $B$, and the extension is
  unique up to unique isomorphisms.

  %  Notation and terminology
  We introduce some notation. When $(C,\pi,\mathbf x;\mathbf y,N,v_1,v_2)$ is
  a family over $B$, we denote by $C_b$ the special fibre $\pi^{-1}(b)$, and by $x_i(b)$ the
  $i$-th heavy marking of $C_b$, and similar notation for the light markings.
  By a rational tail (resp. bridge) $E\subset C_b$ we mean that $E$ is a smooth
  rational subcurve of $C_b$ intersecting the rest of $C_b$ at one (resp. two) node(s) of $C_b$.

  % Quick reductions

  First we reduce to the case where $\pi^*:C^*\to B^*$ is smooth.
  By the standard stable reduction argument, possibly after finite
  base-change, the normalization of $C^*$ is a disjoint union $\coprod_{i=0}^{k}C^*_i$ of smooth curves over
  $B^*$.
  We view the preimages of $\mathbf x^*$ and $\mathbf y^*$ as heavy and light
  markings on $\coprod_{i=0}^{k}C^*_i$. We also view the preimages of the nodes
  of $C^*$ as heavy markings. Assume that the preimage of $y_1^*$ is in $C^*_0$.
  For $i>0$, $C^*_i$ together with the markings forms a
  family $\xi_i^*$ of Hassett's stable curves in some $\overline M_{g_i,m_i|n_i}(B^*)$.
  For $i=0$, the map $C_0^*\to C^*$ induces an isomorphism of
  relative tangent sheaves near $y_1^*$. Thus the pointed curve $C_0^*$ together
  with $(N^*,v_1^*,v_2^*)$ is an object $\xi_0^*$ in $\tilde
  M_{g_0,m_0|n_0}(B^*)$, for some appropriate $g_0,m_0,n_0$.
  For $i>0$, since Hassett's moduli spaces are proper, possibly after finite
  base-change, $\xi^*_i$ extends uniquely to a $B$-family $\xi_i\in \overline M_{g_i,m_i|n_i}(B) $.
  Hence if possibly after finite  base-change $\xi_0^*$ also extends uniquely to a
  $B$-family $\xi_0\in \tilde M_{g_0,m_0|n_0}(B^*)$, by gluing the $\xi_i$ along
  each pair of heavy markings coming from the nodes of $C^*$, we get a
  unique extension of $\xi^*$ to $\xi\in \tilde M_{g,m|n}(B)$. Hence without loss of generality, we assume that
  $\pi^*$ is smooth.

  If $v^*_1= 0$ or $v_2^*= 0$, then $\xi^*$ is equivalent to a family of Hassett's
  stable curves with weighted markings.
  The theorem follows from the properness of  Hassett's moduli spaces.
  Hence  we assume $v^*_1\neq 0$ and $v^*_2\neq 0$ on $B^*$.

  % Existence
  We first consider the case $(g,m)\neq (0,1)$. In this case
  $(C^*,\pi^*,\mathbf x^*;\mathbf y^*)$
  is family of Hassett-stable curves with heavy markings $\mathbf x^*$ and light
  markings $\mathbf y^*$. Possibly after base-change we  extend it to a
  $B$-family of Hassett-stable curves $(C,\pi,\mathbf x;\mathbf y)$. We claim
  that $(N^*,v^*_1,v^*_2)$ has a unique extension $(N,v_1,v_2)$ to $B$ such that
  $(v_1,v_2)$ have no common zero. Indeed, by fixing a trivialization of
  $T_{y_1}$, we can identify $T_{y_1}$ with $\mathcal  O_B$. Then $N^*$ together with the
  sections $(v_1^*,v_2^*)$ is equivalent to a map from $B^*$ to $\mathbb P^1$.
  This map has a unique extension to $B$. The
  extension of the map is equivalent to the extension of $(N^*,v^*_1,v^*_2)$ to
  $(N,v_1,v_2)$.

  We now modify the family $(C,\pi,\mathbf x;\mathbf y,N,v_1,v_2)$ by
  iteratively blowing up $C$ at some smooth points
  of the special fibre to make it stable. The only situation that
  violates the stability condition is when
  \begin{equation}
    \label{unstable_1}
    v_1(b)=0  \quad \text{and} \quad  y_1(b) = y_j(b)\text{ for some }j\neq 1.
  \end{equation}

  If this happens, let $q:C^\prime\to C$ be the blowup of $C$ at $y_1(b)$, and
  extend $(\mathbf x^*;\mathbf y^*,N^*,v^*_1,v^*_2)$ to a new $B$-family
  $\xi^\prime = (C^\prime,\mathbf x^\prime;\mathbf
  y^\prime,N^\prime,v_1^\prime,v^\prime_2)$.
  We claim that $v_2^\prime(b)\neq 0$ and the vanishing order of $v_1^\prime$ at
  $b$ is
  exactly one less than that of $v_1$. To see this, notice that the map $q$
  induces $q^*\Omega_{C/B} \cong \Omega_{C^\prime/B}(-E)$, where $E$ is the exceptional
  divisor. Hence we have an isomorphism $(\Omega_{C/B})^\vee|_{y_1}(-b)\cong
  \Omega_{C^\prime/B}^\vee|_{y_1^\prime}$ of line bundles on $B$, which
  restricts to the identity on $B^*$. Thus the claim follows immediately.
  We replace $\xi$ by $\xi^\prime$, and
  repeat this procedure finitely many times until (\ref{unstable_1}) does not happen. This
  gives us a chain of exceptional divisors $E_1,\cdots,E_k$. Then we blow down
  the maximal subchains of $E_1,\cdots,E_k$ that does not contain any
  $y_j(b)$ for $j=1,\dots,n$.  This gives a stable family over $B$ and proves
  the existence in the case $(g,m)\neq (0,1)$.

  Now we consider the case $(g,m) = (0,1)$. It follows that $n\geq 2$.
  In this case we can find an $B^*$-isomorphism between $C^*$ and
  $\mathbb P^1\times B^*$, identifying  $y^*_1$
  with $\{0\}\times B^*$, $x^*_1$ with $\{\infty\}\times B^*$ and
  $v_1/v_2$ with the standard tangent vector $\partial/\partial z$, where $z$ is
  the coordinate on $\mathbb P^1$. We first take $C=\mathbb P^1\times
  B$ and $N=\mathcal  O_B$. We set $v_2\equiv 1$ and
  $v_1\equiv \partial/\partial z$.
  Then we take the limit of the markings $y_2^*,\cdots,y_n^*$ to get a family
  $(C,\pi,x_1;\mathbf y,N,v_1,v_2)$ over $B$.
  The only situation that violates the stability condition is when
  \begin{equation}
    \label{unstable_2}
    x_1(b) = y_j(b) \text{ for some } j\neq 1.
  \end{equation}
  If this happens, we blow up $C$ at $x_1(b)$ to get a new family.
  In the new family $x_1(b)$ lies on the exceptional divisor. If $x_1(b)$ still
  intersects some $y_j(b)$, we repeat this procedure until (\ref{unstable_2}) does not happen.
  This only takes finitely many steps, since in the generic
  fibre $x_1$ does not intersect any light markings.
  After these blowups, the special fibre is a chain of smooth rational curves
  $E_1,\cdots,E_k$, where $y_1(b)\in E_1,x_1(b)\in E_k$ and $E_i$ intersects
  $E_{i+1}$ at a node for $i=1,\cdots,k-1$.
  Then we contract all the $E_i$ that does not contain any $y_j(b)$ for $j=2,\dots,n$.
  This gives a stable family over $B$ and proves the existence in
  the case $(g,m)=(0,1)$.

  The following picture shows the change of the special fibre in a typical
  $(g,m)=(0,1)$ case. The subcurves drawn vertically are contracted.
  \begin{figure}[h]
    \centering
    \includegraphics[scale=0.4]{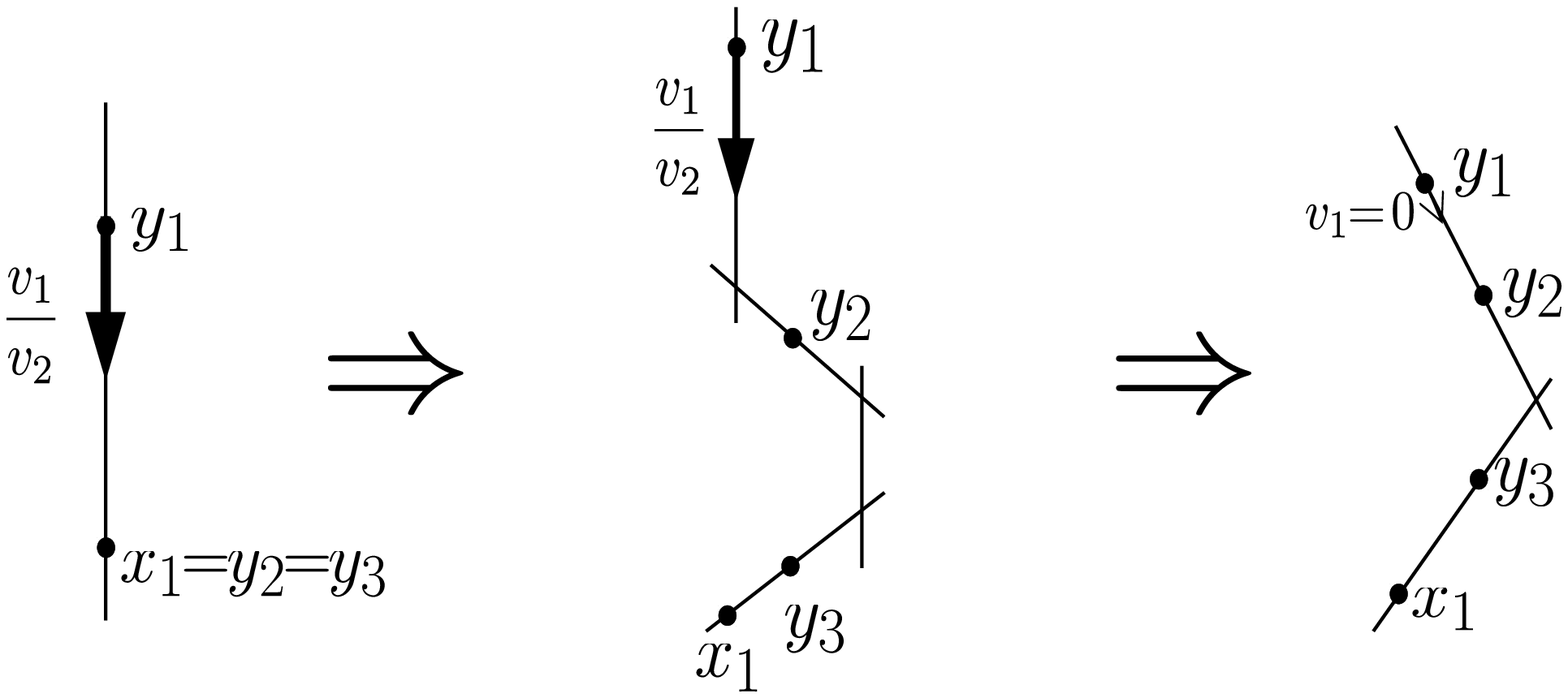}.
  \end{figure}

  % Uniqueness
  For uniqueness,
  let $\xi = (C,\pi,\mathbf x;\mathbf y,N,v_1,v_2)$
  and $\xi^\prime = (C^\prime,\pi^\prime,\mathbf x^\prime;\mathbf y^\prime,N^\prime,v^\prime_1,v^\prime_2)$
   in $\widetilde M_{g,m|n}(B)$ be two extensions of $\xi^*$, we need to show that these two extensions are
  isomorphic, possibly after finite  base-change.
  The proof is standard. We construct a third (unstable) $B$-family
  $\tilde\xi = (\widetilde C,\widetilde \pi,\widetilde{\mathbf x};\widetilde{\mathbf y},\widetilde N,\widetilde
  v_1,\widetilde v_2)$ dominating $\xi$ and $\xi^\prime$, and inducing
  isomorphisms over $B^*$.
  Here we are only assuming that $\tilde C$ is a family of nodal curves, all markings are
  in the relative smooth locus and $\widetilde v_1,\widetilde v_2$ have no common zeros.
  We may also assume that $\widetilde C$ is a regular surface.
  One can check that due to the stability condition the maps $q:\widetilde C\to C$ and $q^\prime:\widetilde C\to C^\prime$ contract
  the same set of rational subcurves in the special fibre, thus it induces an
  isomorphism between $\xi$ and $\xi^\prime$. We will skip the details here.
\end{proof}
\begin{remark}
We can actually prove that $\widetilde{M}_{g,m|n}$  has a projective coarse
moduli space. Since we will not need this result, we only sketch the proof here.

We apply induction on $n$. When $n=1$, $y_1$ is  is
not allowed to collide with any other markings since all other markings are
heavy. Thus $y_1$ is equivalent to a heavy marking. Hence $\tilde M_{g,m|1}$ is
isomorphic to the projective bundle $\mathbb P_{\overline M} (T_{y_1}\oplus
\mathcal O_{\overline M})$, where $\overline
M=\overline{M}_{g,m+1}$. Since $\overline{M}_{g,m+1}$ has a projective coarse
moduli, so does $\tilde M_{g,m|1}$.

Let $\widetilde{\mathcal  C}_{g,m|n}$ be the universal
curve over $\widetilde{M}_{g,m|n}$. If $\widetilde{M}_{g,m|n}$  has a projective
coarse moduli, so does $\widetilde{\mathcal  C}_{g,m|n}$.
For each $n\geq 1$, we look at the morphism $\tau_n:\widetilde{M}_{g,m|n+1}\to
\widetilde{\mathcal  C}_{g,m|n}$ defined by viewing the last light marking
$y_{n+1}$ as the distinguished point on the curve and then stabilizing the
curve. It suffices to prove that $\tau_n$ is projective. The fibres of $\tau_n$ are
either points or isomorphic to $\mathbb P^1$. The locus $Z\subset \widetilde{M}_{g,m|n+1}$
where $y_1$ coincides with $y_{n+1}$ is a divisor, which intersects every one
dimensional fibre at exactly one point. Hence $\mathcal
O_{\widetilde{M}_{g,m|n+1}}(Z)$ is ample when restricted to the fibres. Note
that we can generalize \cite[Tag 0D2S]{stacks-project} to the case of algebraic spaces since
we have the formal function theorem for algebraic spaces
\cite{knutson1971algebraic}. This implies that $\mathcal
O_{\widetilde{M}_{g,m|n+1}}(Z)$ is relatively ample, thus $\tau_n$ is projective.

\end{remark}

\section{The moduli of stable $r$-spin  curves with mixed weighted markings}
\label{r_spin_master}
In this section we define the moduli space $\tilde M^{1/r}_{g,\gamma}$ of genus-$g$ stable
$r$-spin curves with mixed weighted markings.
Then we introduce a $\mathbb C^*$-action on
$\tilde M^{1/r}_{g,\gamma}$ and study the
fixed-point components.
\subsection{The moduli space $\tilde M^{1/r}_{g,\gamma}$ and its properness}
In this subsection we assume that $2g-1+m\geq 0$, $n\geq 1$ and
$(2g-1+m,n-1)\neq (0,0)$.
As before let
\[\gamma =
  \big(\frac{a_1}{r},\dots,\frac{a_m}{r}\big|\frac{b_1}{r},\dots,\frac{b_n}{r}\big),
  \quad
a_i,b_j \in \{1,\cdots,r\}.
\]
\begin{definition}
  \label{defn_spin_master}
  An $S$-family of stable $r$-spin curves with {\it mixed} $\gamma$-weighted
  markings is the datum
  \[
    \xi=(C,\pi,x_1,\cdots,x_m;y_1,\cdots,y_n,N,L,v_1,v_2,p),
  \]
  where
  $(C,\pi,\mathbf x;\mathbf y,L,p)$ is an $S$-family of genus-$g$ pre-stable
  curves  with $\gamma$-weighted markings (Definition \ref{weighted_r_spin}),
  such that if $\rho:C\to |C|$ is the coarse moduli of $C$ and $|\pi|:|C|\to S$ the
  induced projection, then
  \[
    \big(
    |C|,|\pi|,\rho(\mathbf x);\rho(\mathbf y),N,(d\rho_{y_1}\otimes
    \mathrm{id}_N)(v_1),v_2\big)    \in \widetilde
    M_{g,m|n}(S).
  \]
\end{definition}
It is obvious how to pullback $\xi$ along any $S^\prime \to S$, as in
Section \ref{subsec_curves}. This defines the category $\widetilde
M_{g,\gamma}^{1/r}$ of stable genus-$g$ $r$-spin curves with mixed $\gamma$-weighted
markings.  We assume that the degree constraint (\ref{eq:selection}) is
satisfied so the category is not empty.
\begin{theorem}
  \label{proper_spin_master}
  The category $\widetilde M_{g,\gamma}^{1/r}$ is a smooth proper
  Deligne--Mumford stack.
  % with projective coarse moduli.
\end{theorem}
\begin{proof}
  The proof of properness is essentially the same as the proof of Theorem 1.5.1 in
  \cite{10.2307/1194469}.
  Let $\pi:\mathcal  C\to \widetilde
  M_{g,m|n}$ be the universal curve with universal markings
  $\mathbf x;\mathbf y$.  Let $\mathcal C_{\mathrm{rt}}$ be the stack of $r$-th roots of
  the line bundle
  \[
    \omega_{\mathcal  C/\widetilde M_{g,m|n}}\big( \sum_{i=1}^m x_i +
    \sum_{b_j\neq r} (1-b_j)y_j + \sum_{b_j=r} y_j\big).
  \]
  This is a proper Deligne--Mumford stack. Let $\beta$ be the class of a fibre
  of $\mathcal C_{\mathrm{rt}}\to \widetilde M_{g,m|n}$. We consider  $\mathcal
  K^{\mathrm{bal}}_{g,m}(\mathcal  C_{\mathrm{rt}}/\widetilde M_{g,m|n},\beta)$, the stack of balanced twisted stable maps of genus $g$
  and class $\beta$ into $\mathcal  C_{\mathrm{rt}}$ relative to $\widetilde
  M_{g,m|n}$. Then $\widetilde M_{g,\gamma}^{1/r}$ is isomorphic to the closed substack
  of $\mathcal K^{\mathrm{bal}}_{g,m}(\mathcal  C_{\mathrm{rt}}/\widetilde
  M_{g,m|n},\beta)$ where the $i$-th marking is mapped to $x_i$, for all $i$.
  This is a proper Deligne--Mumford stack \cite{abramovich2002compactifying}.

  The proof of smoothness is  identical to the proof of Proposition 2.1.1
  in  \cite{10.2307/1194469}.
\end{proof}

\subsection{The $\mathbb C^*$-action on $\widetilde M_{g,\gamma}^{1/r}$ and the fixed-point components}
\label{C_star_spin}
From now on we assume that $n\geq 1$ and $2g-2+m\geq 0$.
We introduce an $\mathbb C^*$-action on
$\widetilde M^{1/r}_{g,\gamma}$ via
\begin{equation}
  \label{action}
  \lambda \cdot (C,\pi,\mathbf x;\mathbf y,N,L, v_1,v_2,p) = (C,\pi,\mathbf x;\mathbf y,N,L,\lambda v_1,v_2,p),\quad
  \lambda\in \mathbb C^*.
\end{equation}
This also defines an action on the universal family $(\mathcal  C,\pi,\mathbf x;\mathbf y,\mathcal  N,v_1,v_2,\mathcal  L,p)$:
it defines a $\mathbb C^*$-action on $\mathcal  C$; $\mathcal  L$  is an
equivariant line bundle on $\mathcal  C$; $\mathcal  N$ is an
equivariant line bundle on $\widetilde M^{1/r}_{g,m|n}$; all the markings, $p$ and $(v_1,v_2)$
are preserved by these actions.

The fixed-point components of $\tilde M^{1/r}_{g,\gamma}$ are
\begin{enumerate}
\item
  $F^{1/r}_0 = \{\xi:v_1 = 0\}$;
\item
  $F^{1/r}_\infty = \{\xi:v_2 = 0\}$;
\item
  For each $J \subset \{1,\cdots,n\}$ such that $\{1\}\subsetneqq J$,
  $F^{1/r}_J$ consists of
  \[
    \xi=(C,\pi,\mathbf x;\mathbf y,N,L,v_1,v_2,p)
  \]
  such that
  \begin{itemize}
  \item $C=C_J\cup E$ where $E$ is a smooth rational orbifold intersecting
    $C_J$ at a single node;
  \item $v_1\neq 0$ and $v_2\neq 0$;
  \item $y_1\in E$ and $y_j=y_1$ for all $j\in J$.
  \item all other markings are on $C_J$.
  \end{itemize}
\end{enumerate}
It is easy to see that these are all the set-theoretic fixed-point components.
Indeed, suppose $\xi \in \tilde
M_{g,m|n}(\mathbb C)$ is $\mathbb C^*$-fixed and $v_1\neq 0,v_2\neq 0$. Let  $E$ be the
subcurve  containing $y_1$. For each $1\neq \lambda\in \mathbb C^*$, the isomorphism
between $\xi$ and $\lambda\cdot\xi$ induces a nontrivial automorphism of $E$
fixing all the markings and nodes of $C$ contained $E$. It follows that $E$
is a smooth rational orbifold intersecting the remainder of $C$ at a single node, and all the
markings on $E$ coincide with $y_1$. Hence $\xi$ is in some
$F_J^{1/r}$. A general $\xi$ in $F^{1/r}_J$ looks like (the $r$-spin structure
is not drawn):

\begin{figure}[h]
\centering
  \includegraphics[scale=.12]{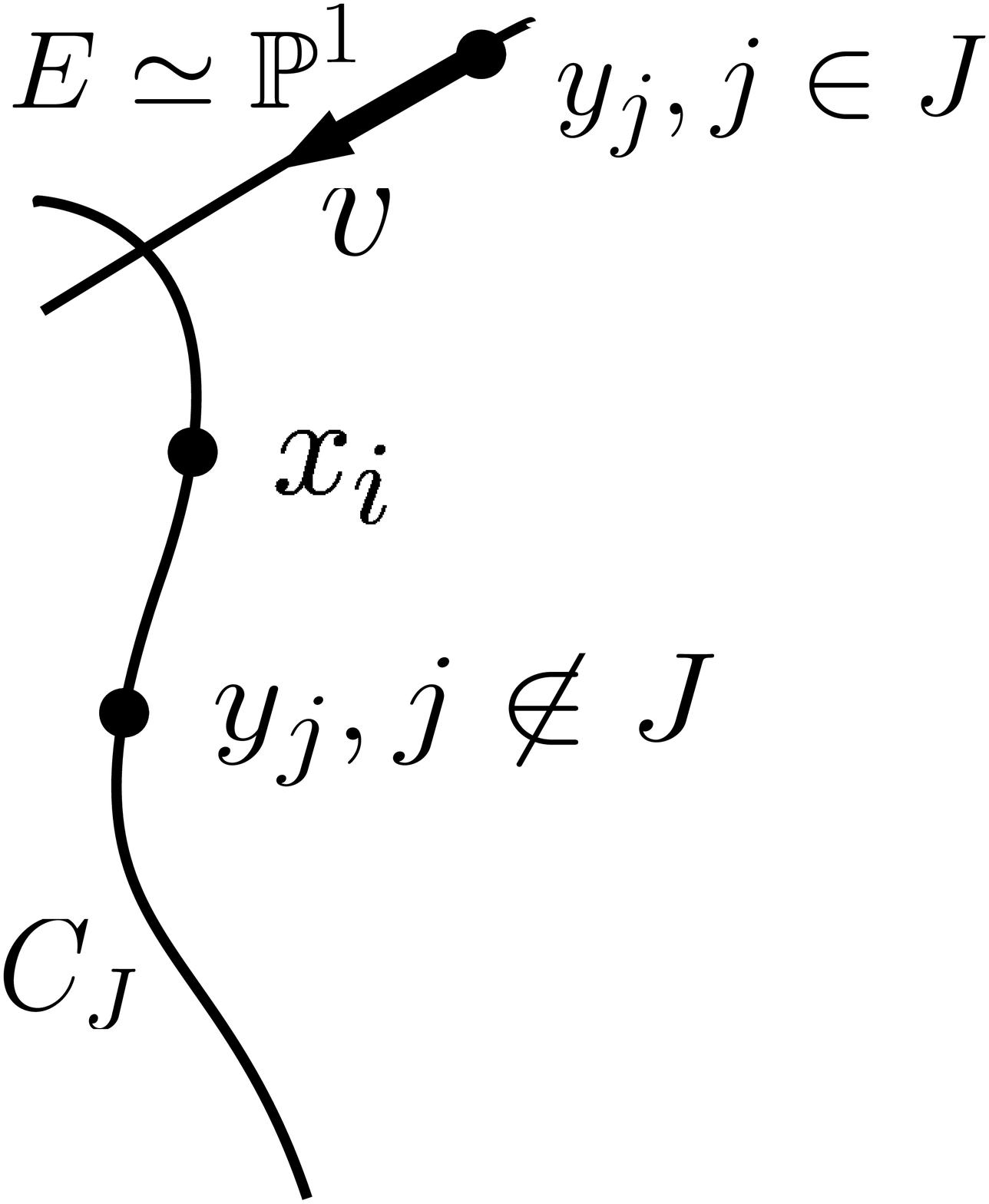}.
\end{figure}

We introduce some notation before describing the stack structure and the normal
bundle of each fixed-point component.
For any integer $w$, let $\mathbb C_w$ be the standard $\mathbb C$ with the
$\mathbb C^*$-action of weight $w$:
\[
  \lambda \cdot z = \lambda^{w}z, \quad \lambda \in \mathbb C^*,z\in \mathbb C.
\]
Let
\[
  \gamma^\prime =
  \big(\frac{a_1}{r},\cdots,\frac{a_m}{r},\frac{b_1}{r}\big|\frac{b_2}{r},\cdots,\frac{b_n}{r}\big).
\]
Recall that $T_{y_1}$ is the line bundle formed by the tangent spaces to the
curves at $y_1$.
\begin{lemma}
  \label{normal_r_spin_0_infty}
  $\phantom{}$
   \begin{enumerate}
   \item $F^{1/r}_0\cong \overline
  M^{1/r}_{g,\gamma^\prime}$, where the heavy markings are $x_1,\cdots,x_m,y_1$ and the light
  markings are $y_2,\cdots,y_n$.
Its equivariant normal bundle is isomorphic to $T_{y_1}\otimes \mathbb C_1$.
\item $F^{1/r}_\infty\cong \overline M^{1/r}_{g,\gamma}$, where the heavy
  markings are all the $x_i$ and
  the light markings are all the $y_j$. Its equivariant normal bundle is isomorphic to $(T_{y_1}\otimes \mathbb C_1)^\vee$.
   \end{enumerate}
\end{lemma}
\begin{proof}

  We prove (1); (2) is similar. Since $v_1$ vanishes identically on $F^{1/r}_0$,
  $y_1$ does not intersect any other
  markings there. To prove $F^{1/r}_0\cong \overline
  M^{1/r}_{g,\gamma^\prime}$, we only need to identify the non-orbifold marking $y_1$ of $F^{1/r}_0$ with the
  last orbifold heavy marking of $\overline M^{1/r}_{g,\gamma^\prime}$. This is worked out
  in Section \ref{other_description}. To compute the normal bundle, note that
  near $F^{1/r}_0$, $v_2$ is never vanishing. Hence $N$ is trivialized by $v_2$
  and thus the datum $(N,v_1,v_2)$ is equivalent to $v_1/v_2$ (cf. Remark \ref{rmk_effect_v_2}).
  Hence near $F^{1/r}_0$, $\tilde M_{g,\gamma}^{1/r}$ is the total space of
  $T_{y_1}$ over $F^{1/r}_0$. The $\mathbb C^*$-action is identified with
  fibrewise multiplication. Hence the formula for the normal bundle follows.
\end{proof}

We now come to $F^{1/r}_J$. Let $\xi=(C,\pi,\mathbf x;\mathbf y,N,L,v_1,v_2,p)$
be any closed point of $F^{1/r}_J$. Recall that $C=C_J\cup E$, where $E$ is a smooth
rational orbifold intersecting $C_J$ at a node. The first observation is
that the orbifold index of the node $C_J\cap E$ and the monodromy of the line bundle $L$ at the node
 are uniquely determined by the data $\gamma$ and $J$.
Let $k,a_\infty\in \{1,\cdots,r\}$ and $\ell \in \mathbb Z_{\geq 0}$  be the integers such that
\begin{equation}
  \label{k_and_l}
  r\ell +k =1 + \sum_{j\in J} (b_j-1)    \quad \text{and} \quad
  a_\infty \equiv -k \mod r,
\end{equation}

Let $r^\prime = r/\gcd(r,k)$.
\begin{lemma}
  The orbifold index at the node $C_J\cap E$ is $r^\prime$. The monodromy of
  $L|_{C_J}$ at the node is $\frac{k}{r}$. The monodromy of
  $L|_{E}$ at the node is $\frac{a_\infty}{r}$.
\end{lemma}
\begin{proof}
  We have an isomorphism of line bundles
  \[
    (L|_{E})^r \cong \omega_C|_E\! ~\big(\!\!\!\!\!\sum_{j\in J,b_j\neq r}(1-b_j)y_j +
    \!\!\!\!
    \sum_{j\in J, b_j=r}y_j\big).
  \]
  Hence
  \[
    r\deg(L|_{E})     \equiv \sum_{j\in J}(1-b_j) -1 \mod r .
  \]
  Since the node is the only
  orbifold point of $E$, the fractional part of $\deg(L|_E)$ must come from the
  orbifold structure at the node.
  The monodromy of $L|_{E}$ at the node follows from the fact that $L$ is
  representable; the monodromy of $L|_{C_J}$ at the node follows from the fact
  that the node is balanced.
  \end{proof}

The stack structure of $F^{1/r}_J$ is more subtle due to the existence of the
``ghost automorphisms''.
We set
\[
  \gamma_J =
  \big(\frac{a_1}{r},\cdots,\frac{a_m}{r},\frac{k}{r}\big|(\frac{b_j}{r})_{j\not\in
  J}\big)
\quad \text{and} \quad
  \gamma_{E} = \big(\frac{a_\infty}{r}\big|\frac{k}{r}\big).
\]
Consider the moduli spaces
$\overline{M}^{1/r}_{g,\gamma_J}$ and $\tilde M^{1/r}_{0,\gamma_E}$, where
the markings are labeled as
  $(x_1,\cdots,x_m,x_J;(y_j)_{j\not\in J})$ and $(x_\infty;y_1)$
respectively.

We want to glue the two gerbes $x_J$ and $x_\infty$ to get a balanced node. For
any scheme $S$, we define a category $\mathcal I(S)$. An object of
$\mathcal  I(S)$ is a $4$-tuple $(\Sigma,q,L,s)$, where $q:\Sigma\to S$ is a gerbe banded by
$\pmb\mu_{r^\prime}$; $L$ is a line bundle on $\Sigma$ with monodromy
$\frac{k}{r}$; $s$ is a non-vanishing section of $L^r$. An arrow
$(\Sigma,q,L,s)\to (\Sigma^\prime,q^\prime,L^\prime,s^\prime)$ is a pair
$(f,\eta)$, where $f:\Sigma \to \Sigma^\prime$ is an $S$-isomorphism of banded gerbes, and  $\eta:L\to
f^*L^\prime$ is an isomorphism of line bundles that takes $s$ to $f^*s^\prime$.
A priori, $\mathcal  I(S)$ is only a $2$-category. But representablity of $L$ implies that
the automorphism group of any arrow is trivial. Hence $\mathcal  I(S)$ is equivalent
to a category.
\begin{lemma}
  \label{lem_inertia}
  The category $\mathcal  I(S)$ is naturally equivalent to
  $B\pmb\mu_{r/r^\prime}(S)$.
\end{lemma}
\begin{proof}
  We have an isomorphism of stacks $B\pmb\mu_{r/r^\prime}\cong [\mathbb C^*/\mathbb C^*]$,
  where the action of $\lambda\in \mathbb C^*$ is multiplication by
$\lambda^{r/r^\prime}$. Using this isomorphism, an object in $B\pmb\mu_{r/r^\prime}(S)$ is a line bundle $M$ on $S$
  and a non-vanishing section $t$ of $M^{r/r^\prime}$.
  Given $(\Sigma,q,L,s)$, the line bundle $L^{r^\prime}$ has trivial monodromy, hence descends to a
  line bundle $M$ on $S$. The section $s\in H^0(\Sigma,
  (L^{r^\prime})^{r/r^\prime})$
  descends to a non-vanishing section $t$ of $M^{r/r^\prime}$. Conversely,
  given $(M,t)$, let $q:\Sigma\to S$ be the stack of $r^\prime$-th roots of $M$,
  and $L^{r^\prime}\cong q^* M$ be the universal $r^\prime$-th root. The
  section $t$ pulls back to a non-vanishing section $s$ of $L^r$. We can check
  that these two procedures are inverses to each other and define an equivalence of categories.
\end{proof}
\begin{remark}
  Actually $B\pmb\mu_{r/r^\prime}$ is a connected component of the rigidified inertia stack
  $\overline{\mathcal  I}_{r^\prime}B\pmb\mu_r$, which classifies gerbes
  $\Sigma$ banded by $\pmb\mu_{r^\prime}$ with representable $\Sigma
  \to B\pmb \mu_r$ \cite{abramovich2008lectures}.
\end{remark}

Consider any $S$-families
\[
  \xi_J = (C_J,\pi_J,x_1,\cdots,x_m,x_J;(y_j)_{j\not\in J},L_J,p_J) \in
  \overline{M}^{1/r}_{g,\gamma_J}(S)
\]
and
\[
  \xi_E = (E,\pi_E,x_\infty;y_1,N,L_E,v_1,v_2,p_E) \in \tilde M^{1/r}_{0,\gamma_E}(S).
\]
By definition, the marking $x_J$ is a gerbe banded by $\pmb\mu_{r^\prime}$ over
$S$. The restriction of the line bundle
\[
  \omega_{C_J/S}\big(x_J+\sum_{i=1}^m x_i+ \sum_{j\not\in J, b_j\neq r}
  (1-b_j)y_j + \sum_{j\not\in J, b_j= r} y_j \big)
\]
to $x_J$ is canonically trivial. Hence $p_J|_{x_J}$ is a
non-vanishing section of $L_J^r|_{x_J}$.
By Lemma \ref{lem_inertia}, this defines an
``evaluation'' map
\[
  \mathrm{ev}_{x_J}: \overline{M}^{1/r}_{g,\gamma_J} \longrightarrow B\pmb\mu_{r/r^\prime}.
\]
Similarly, we have an evaluation map $ \widehat{\mathrm{ev}}_{x_\infty}$ by ``evaluating'' at $x_\infty$ but
reversing the banding of $x_\infty$
\[
  \widehat{\mathrm{ev}}_{x_\infty} : \tilde M^{1/r}_{0,\gamma_E} \longrightarrow  B\pmb\mu_{r/r^\prime}.
\]

We now define a morphism
\[
  \imath_J:\overline{M}^{1/r}_{g,\gamma_J}\underset{B\pmb\mu_{r/r^\prime}}{\times}
  \tilde M^{1/r}_{0,\gamma_E} \longrightarrow \tilde M^{1/r}_{g,\gamma},
\]
where the fibre product is formed via $\mathrm{ev}_{x_J}$  and  $\widehat{\mathrm{ev}}_{x_\infty}$.
Consider any $S$-families $\xi_J$ and $\xi_E$ as above. We modify $\xi_E$ to get a
  new family $\xi_E^\prime$ as follows.
First let
  \[
    L_E^\prime = L_E (-cy_1),
  \]
  where
  \[
    c = \begin{cases}
      \ell - \big|\{j\in J:b_j =r\}\big| \cond{k\neq r}\\
      \ell - \big|\{j\in J:b_j =r\}\big|+1 \cond{k= r}\\
    \end{cases}.
  \]
  Then for each $j\in J$, set $y_j=y_1$.  We have natural isomorphisms
  \[
    L_E^{-r}\otimes
    \omega_{E/S}(x_\infty+(1-k)y_1)
    \cong (L_E^\prime)^{-r}\otimes \omega_{E/S}(x_\infty+\!\!\!\!\sum_{j\in J,b_j\neq
      r}\!\!\!(1-b_j)y_j + \!\!\!\!
    \sum_{j\in J,b_j=r}
    \!\!\!
    y_j) ,\text{  if  }k\neq r,
  \]
  or
  \[
    L_E^{-r} \otimes
    \omega_{E/S}(x_\infty+y_1)
    \cong
    (L_E^{\prime})^{-r}\otimes \omega_{E/S}(x_\infty+
    \!\!\!
    \sum_{j\in J,b_j\neq r}\!\!\!(1-b_j)y_j + \!\!\!
    \sum_{j\in J,b_j=r}
    \!\!\!
    y_j),\text{  if  }k= r.
  \]
  Let $p^\prime_E$ be the image of $p_E$ under either isomorphism, depending on
  whether $k=r$. We get a new family
\begin{equation*}
  \xi_E^\prime = (C_E,\pi_E,x_\infty;(y_j)_{j\in J},N,L^\prime_E,v_1,v_2,p^\prime_E).
\end{equation*}
By Lemma \ref{lem_inertia}, an $S$-point of
$\overline{M}^{1/r}_{g,\gamma_J}\underset{B\pmb\mu_{r/r^\prime}}{\times}
\tilde M^{1/r}_{0,\gamma_E}$  consists of $(\xi_J, \xi_E)$ as above, and an $S$-isomorphism
between
\[
  \theta_J=(x_J,\pi_J|_{x_J},L_J|_{x_J},p_J|_{x_J})
  \quad \text{and} \quad
  \theta_E = (x_\infty,\pi_{E}|_{x_\infty},L_E|_{x_\infty},p_E|_{x_\infty}).
\]
Since to get $L^\prime_E$ we have only modified $L_E$ near $y_1$,
$\theta_E$ is naturally isomorphic to
\[
  \theta_E^\prime =
  (x_\infty,\pi_{E}|_{x_\infty},L^\prime_E|_{x_\infty},p^\prime_E|_{x_\infty}).
\]
We use the isomorphism $\theta_J\cong \theta^\prime_E$ to glue $\xi_J$ and
$\xi^\prime_E$ along $x_J$ and $x_\infty$, and get
\[
  \xi = (C,\pi,x_1,\cdots,x_m;y_1,\cdots,y_n,N,L,v_1,v_2,p) \in \tilde M^{1/r}_{g,\gamma}(S).
\]
This defines the morphism $\imath_J$.
For more about gluing stacks and their morphisms, see \cite{abramovich2008gromov}.
  See also Section 2.3 in \cite{Chiodo:2009hbu} on gluing the spin structures.

\begin{lemma}
  \label{stack_F1/r_J}
  The morphism $\imath_J$ induces an isomorphism onto the substack $F^{1/r}_J$.
\end{lemma}
\begin{proof}
An automorphism of $\xi$ consists of an automorphism of $\xi_J$ and automorphism
of $\xi_E^\prime$ that respect the identification of $\theta_J$ and
$\theta^\prime_E$. Moreover, the automorphism groups of $\xi^\prime_E$
and $\xi_E$  are naturally isomorphic. Hence the automorphisms of $\xi$ are
precisely the automorphisms of $(\xi_J,\xi_E,\theta_J\cong \theta_E)$ in the
fibre product of stacks.
Hence the morphism $\imath_J$ induces an
isomorphisms of automorphism groups.

It is easy to see that $\imath_J$ induces a bijection of closed points onto
$F_J^{1/r}$. Note that $F_J^{1/r}$ is smooth since it is a fixed-point
component in a smooth stack. Since $\imath_J$ is representable and proper, it
must be an isomorphism onto $F_J^{1/r}$.
\end{proof}

We now compute the equivariant normal bundle of $F_{J}^{1/r}$.
Recall that (\ref{action}) defines the $\mathbb C^*$-action on $\tilde
M^{1/r}_{0,\gamma_E}$
\[
  \lambda\cdot (E,\pi_E,x_\infty;y_1,L,N,v_1,v_2,p) =
  (E,\pi_E,x_\infty;y_1,L,N,\lambda v_1,v_2,p), ~\lambda \in \mathbb C^*.
\]
This also defines an action on the universal curve $\mathcal  E$ over
$\tilde M^{1/r}_{0,\gamma_E}$. Moreover, the map $\mathrm{pr}^*_2\mathcal  E\to \mathcal  C$
induced by $\imath_J$ is
equivariant, where $\mathcal  C$ is the universal curve over $\tilde
M_{g,\gamma}^{1/r}$, and $\mathrm{pr}_2:\overline{M}^{1/r}_{g,\gamma_J}\underset{B\pmb\mu_{r/r^\prime}}{\times}
\tilde M^{1/r}_{0,\gamma_E}\to \tilde M^{1/r}_{0,\gamma_E}$ is the second projection.

We now describe the $\mathbb C^*$-action on $\mathcal  E$. Note that $\tilde M^{1/r}_{0,\gamma_E}\cong B\pmb
\mu_r$ and the $\mathbb C^*$-action on $\tilde M^{1/r}_{0,\gamma_E}$ is trivial.
The stability condition requires $v_2$ to be non-vanishing.
Hence $(N,v_1,v_2)$ is equivalent to a tangent vector $v_1/v_2$ at $y_1$. Let
$z$ be the coordinate on the coarse moduli $|\mathcal  E|$ of $\mathcal E$ so that $y_1$ is at
$z=0$ and $x_\infty$ is at $z=\infty$. Then the induced action on $|\mathcal
E|$  is given by
\begin{equation}
  \label{action_on_P1}
  \lambda \cdot z = \lambda^{-1} z, ~\lambda \in \mathbb C^*.
\end{equation}

Let $T_{x_J}$ (resp. $\!T_{x_\infty}$) be line bundle on
$\overline M^{1/r}_{g,\gamma_J}$ (resp. $\tilde M^{1/r}_{0,\gamma_E}$) formed by the
tangent spaces of the coarse curves along the marking $x_J$ (resp. $x_\infty$).
Recall that for an integer $w$, $\mathbb C_w$ is the standard $\mathbb C$ with
the $\mathbb C^*$-action of weight $w$.

\begin{proposition}
  \label{normal_spin}
  The normal bundle of
  $F_J^{1/r}$ in $\tilde M^{1/r}_{g,\gamma}$ is isomorphic to
  \[
    \mathcal   N_{\mathrm{node}} \oplus \big(\mathcal  O_{F_J^{1/r}}\otimes_{\mathbb C}
    \mathbb C_{-1}\big)^{\oplus(|J|-1)}.
  \]
  where $\mathcal  N_{\mathrm{node}}$ is a line bundle such that $\mathcal N_{\mathrm{node}}^{r^\prime}\cong T_{x_J} \boxtimes
  T_{\infty}$, and $|J|$ is the cardinality of $J$.
\end{proposition}
\begin{proof}
  As in the proof of Proposition 2.1.1 in \cite{10.2307/1194469}, the
  deformation and obstruction of stable $r$-spin curves with mixed $\gamma$-weighted
  markings are identical to those of the underlying twisted curves with mixed weighted markings.
  There is no obstruction in our case since $\tilde M^{1/r}_{g,\gamma}$ is smooth. The deformation of $\xi\in\tilde M^{1/r}_{g,\gamma}$ consists of two parts:
  the deformation of the underlying twisted curves and the deformation of the
  ``tangent vector'' $(N,v_1,v_2)$ (cf. Definition \ref{defn_spin_master}).
  For $\xi\in F^{1/r}_J$ the underlying twisted curve have infinitesimal automorphisms that
  ``cancel'' with the deformation of $(N,v_1,v_2)$.

  We first study the underlying twisted curves. Let $\mathfrak M^{\mathrm{tw}}$ be the
  Artin stack of twisted curves with balanced nodes and
  not necessarily distinct markings $x_1,\cdots,x_m,$ $y_1,\cdots,y_n$.
  Let $\mathfrak Z^{\mathrm{tw}}\subset \mathfrak M^{\mathrm{tw}}$ be the closed
  substack where the curve has a node separating $\{y_j\}_{j\in J}$ and $\{x_1,\cdots,x_n\}\cup
  \{y_j\}_{j\not\in J}$, and $\mathfrak M^{\prime,\mathrm{tw}}\subset \mathfrak
  M^{\mathrm{tw}}$ be the closed substack where $y_j$ is equal to  $y_1$ for all
  $j\in J$. Let $\mathfrak Z^{\prime,\mathrm{tw}}  = \mathfrak Z^{\mathrm{tw}}
  \times_{\mathfrak M^{\mathrm{tw}}}\mathfrak M^{\prime,\mathrm{tw}}$.
  Let $\mathfrak M,\mathfrak
  Z,\mathfrak M^{\prime},\mathfrak Z^\prime$ be the similarly defined moduli spaces of
  non-orbifold curves.

  The local structure of $\mathfrak M^{\mathrm{tw}}$ was studied in
  \cite{olsson2007}. The versal deformation of a twisted curve is in  Remark
  1.11 of \cite{olsson2007}.
  It follows that the substacks $\mathfrak M^{\prime,\mathrm{tw}}$ and $\mathfrak Z^{\rm{rw}}$
  intersect transversely along $\mathfrak   Z^{\prime,\mathrm{tw}}$.  Hence we
  have an isomorphism of normal bundles on $\mathfrak Z^{\prime,\mathrm{tw}}$
  \[
    \mathcal{N}_{\mathfrak  Z^{\prime,\mathrm{tw}}/\mathfrak  M^{\mathrm{tw}}}
    \cong
    \mathcal{N}_{\mathfrak Z^{\mathrm{tw}}/\mathfrak  M^{\mathrm{tw}}}
    \oplus
    \mathcal{N}_{\mathfrak M^{\prime,\mathrm{tw}}/\mathfrak  M^{\mathrm{tw}}}.
  \]
  For simplicity, here (and after) we have suppressed the pullback notation.

  We want to describe those normal bundles in terms of line bundles on the moduli of coarse curves.
  From the description of the versal deformation, we see that under the forgetful morphism
  $\mathfrak M^{\mathrm{tw}}\to \mathfrak M$, the smooth divisor $\mathfrak Z$ pulls
  back to $r^\prime \mathfrak Z^{\mathrm{tw}}$. Let $\mathcal  N_{\mathrm{node}}
  $ be $\mathcal{N}_{\mathfrak Z^{\mathrm{tw}}/\mathfrak  M^{\mathrm{tw}}}$, then we have an isomorphism on $\mathfrak Z^{\mathrm{tw}}$
  \[
    (\mathcal  N_{\mathrm{node}})^{r^\prime} \cong
    \mathcal{N}_{\mathfrak Z/\mathfrak M}.
  \]
  Similarly, we have
  \[
    \mathcal{N}_{\mathfrak M^{\prime,\mathrm{tw}}/\mathfrak  M^{\mathrm{tw}}} \cong
    T_{y_1}^{\oplus (|J|-1)},
  \]
  where $T_{y_1}$ is the line bundle formed by the tangent spaces of the curves at $y_1$.

  We claim that the normal bundle of $F^{1/r}_J$ is the pullback of
  $\mathcal{N}_{\mathfrak Z^{\prime,\mathrm{tw}}/\mathfrak M^{\mathrm{tw}}}$ via the forgetful morphism. We
  look at the $\mathbb C^*$-invariant open substack $U\subset \widetilde M^{1/r}_{g,m|n}$
  of $\xi$  where
  $v_1\neq 0$ and $v_2\neq 0$. It contains $F^{1/r}_J$. Over $U$ the data $(N,v_1,v_2)$ is equivalent to
  $v_1/v_2$, a nonzero tangent vector at the light marking $y_1$.
  Since the $r$-spin structure has trivial deformation and obstruction, the forgetful morphisms $U\to \mathfrak M^{\mathrm{tw}}$ and
  $F^{1/r}_J \to \mathfrak Z^{\prime,\mathrm{tw}}$
  induce a distinguished triangles on $U$
  \begin{equation}
    \label{triangle_U}
    T_{y_1}[0] \longrightarrow \mathcal{T}_U[0] \longrightarrow \mathbb{T}_{\mathfrak  M^{\mathrm{tw}}}\overset{+1}{\longrightarrow}
  \end{equation}
  and a distinguished triangle on $F^{1/r}_J$
  \begin{equation}
    \label{triangle_F}
    T_{y_1}[0] \longrightarrow \mathcal{T}_{F^{1/r}_J}[0] \longrightarrow \mathbb T_{\mathfrak Z^{\prime,\mathrm{tw}}} \overset{+1}{\longrightarrow},
  \end{equation}
  where $\mathcal T$ means the tangent bundle and $\mathbb T$ means the tangent complex.
  The distinguished triangle (\ref{triangle_F}) naturally maps to
  the restriction of (\ref{triangle_U}) to $F_J^{1/r}$. Taking the cones gives us
  an isomorphism $\mathcal{N}_{F^{1/r}_J/U} \cong \mathcal{N}_{\mathfrak  Z^{\prime,\mathrm{tw}}/\mathfrak  M^{\mathrm{tw}}}$ on $F^{1/r}_{J}$.

  To complete the proof, observe that on  $F^{1/r}_J$ we have
  \[
    \mathcal{N}_{\mathfrak Z/\mathfrak M} \cong T_{x_J} \boxtimes T_{x_\infty}
    \quad \text{and} \quad T_{y_1} \cong \mathcal
    O_{F^{1/r}_J}\otimes_{\mathbb C}\mathbb C_{-1}.
  \]
  To see the latter, notice that $y_1$ comes from the factor $\tilde M^{1/r}_{0,\gamma_E}\cong B\pmb
  \mu_r$, and $\pmb\mu_r$ acts trivially on the underlying curves.  Hence $T_{y_1}$ is a
  constant line bundle on ${F^{1/r}_J}$. The $\mathbb C^*$-action is
  given by the tangent map of (\ref{action_on_P1}). Hence is has weight $-1$.
\end{proof}

\section{The perfect obstruction theory and localization}
\label{phi_fields_master}
In this section we first define the virtual cycle on the master
space via introducing $\varphi$-fields.
That is parallel to Section \ref{obstruction_theory_single}.
Then we apply virtual localization to get the basic wall-crossing formula.
\subsection{The $\varphi$-fields and equivariant perfect obstruction theory}

As before, we fix non-negative integers $g,m,n$ such that $2g-2+m\geq 0$ and
$n\geq 1$. We also fix
\[\gamma =
  \big(\frac{a_1}{r},\dots,\frac{a_m}{r}\big|\frac{b_1}{r},\dots,\frac{b_n}{r}\big),
  \quad
a_i,b_j \in \{1,\cdots,r\}.
\]

\begin{definition}
  \label{defn_master_varphi}
   An $S$-family of stable $r$-spin  curves with  mixed $\gamma$-weighted markings and
   $\varphi$-fields consists of
  \[
    (C,\pi,x_1,\cdots,x_m;y_1,\cdots,y_n,N,L,v_1,v_2,p,\varphi_1,\cdots,\varphi_s),
  \]
  where
  \[
    (C,\pi,\mathbf x;\mathbf y,N,L,v_1,v_2,p) \in \tilde
    M^{1/r}_{g,\gamma}(S),
  \]
  and for $\alpha=1,\cdots,s$,
  \[
    \varphi_\alpha\in H^0(C,L(D_\alpha)), \quad
   D_{\alpha} = - \sum_{q_\alpha a_i \in \mathbb
     Z}x_i+\sum_{b_j\neq r}\left \lfloor
     q_\alpha (b_j-1) \right \rfloor y_j -\sum_{b_j= r}y_j.
  \]
  \end{definition}
  As in \cite{chang2015witten}, the category of such families is a
  Deligne--Mumford stack
  $\tilde M^{1/r,\varphi}_{g,\gamma}$ of finite type over $\mathbb C$, and we have a
  representable forgetful morphism
  \[
    \tau: \tilde M^{1/r,\varphi}_{g,\gamma} \longrightarrow  \tilde M^{1/r}_{g,\gamma}.
  \]

We now define the perfect obstruction theory on $\tilde M^{1/r,\varphi}_{g,\gamma}$.
We abbreviate
\[
  \tilde M^{1/r} = \tilde M^{1/r}_{g,\gamma} \quad \text{and} \quad
\tilde M^\varphi = \tilde M^{1/r,\varphi}_{g,\gamma} .
\]
Let $\pi:\mathcal  C\to \tilde M^{\varphi}$ be the universal curve,
$\omega_{\pi}$ be the relative
dualizing sheaf on $\mathcal  C$ and $\mathcal L$
be the universal $r$-spin bundle.
By abuse of notation, we will use $x_i,y_j$ and $D_\alpha$ to denote the
divisors on $\mathcal C$ as in Definition \ref{defn_master_varphi}.

As in \cite{chang2015witten}, $\tau$ has an equivariant relative perfect obstruction
theory $\mathbb E_{\tau} \to  \mathbb L_{\tau}$,
where $\mathbb L_\tau $ is the relative cotangent complex of $\tau$ and
\[
  \mathbb E_\tau = \bigg(\bigoplus_{\alpha=1}^sR\pi_*  \mathcal
    L^{w_\alpha}(D_\alpha) \bigg)^\vee.
\]
It admits a cosection
\[
  \sigma:  h^1(\mathbb E_\tau^\vee)\longrightarrow  R^1\pi_*\omega_{\pi}\cong \mathcal  O_{\tilde M^{\varphi}}
\]
defined by
\[
    \sigma(\dot \varphi_1,\cdots,\dot\varphi_s) = p\sum_{\alpha=1}^s
  \dot\varphi_\alpha \partial_\alpha W(\varphi_1,\cdots,\varphi_s).
\]
Here $h^1$ means the cohomology sheaf in degree $1$. The cosection is
well-defined for the same reason as in Section \ref{obstruction_theory_single}.

Since $\tilde M^{1/r}$ is smooth, we get an absolute perfect
obstruction theory on $\tilde M^{\varphi}$
\[
  \mathbb L^\vee_{\tilde M^\varphi} \longrightarrow  \mathbb E^\vee_{\tilde M^\varphi}
\]
which fits into the distinguished triangle
\begin{equation}
  \label{triangle}
  \tau^* \mathcal  T_{\tilde M^{1/r}}[-1] \longrightarrow \mathbb E_{\tau}^\vee \longrightarrow \mathbb E^\vee_{\tilde
    M^\varphi} \overset{+1}{\longrightarrow} .
\end{equation}
Taking cohomology sheaves, we have
\[
  h^1(\mathbb E_{\tau}^\vee )\longrightarrow h^1( \mathbb E^\vee_{\tilde M^\varphi}).
\]
\begin{lemma}
  The cosection $\sigma$ factors through the absolute obstruction sheaf $h^1( \mathbb
  E^\vee_{\tilde M^\varphi})$ and defines an equivariant cosection of the
  absolute obstruction theory.
\end{lemma}
\begin{proof}
  The proof that $\sigma$ factors through the absolute obstruction sheaf is
  exactly the same as in \cite{chang2015witten}. Since the action is defined by
  scaling $v_1$ but the cosection $\sigma$ is independent of $v_1$, $\sigma$ is
  equivariant  (cf. \cite[Lemma 2.10]{chang2015mixed}).
\end{proof}

As in \cite{chang2015witten}, Serre duality implies that the degeneracy locus
$D(\sigma)$ is equal to $\tilde M^{1/r}_{g,\gamma}$, viewed as a closed subset
of $\tilde M^{\varphi}$ where all the $\varphi$-fields are
identically zero.
By \cite{kiem2013localizing} we have a cosection localized equivariant virtual
fundamental class
\[
  [\tilde M^\varphi]^{\mathrm{vir}}_{\mathrm{loc}} \in A^{\mathbb
    C^*}_{\delta(\gamma)}(\tilde M^{1/r}),
\]
where the virtual dimension
\[
  \delta(\gamma) = (3-s+2q)(g-1) + m + n + 1 - \sum_{\alpha=1}^s\Big(
  \sum_{i=1}^m \langle q_\alpha (a_i-1) \rangle + \sum_{j=1}^n \langle
  q_\alpha (b_j-1) \rangle\Big)
\]
can be computed by the orbifold Riemann-Roch formula.
\subsection{Virtual localization on $\tilde M^{1/r,\varphi}_{g,\gamma}$}

The $\mathbb C^*$-action (\ref{action}) on $\tilde M^{1/r}_{g,\gamma}$ lifts to
$\tilde M^{1/r,\varphi}_{g,\gamma}$ by scaling the $v_1$ of $\tilde M^{1/r,\varphi}_{g,\gamma}$.
The forgetful morphism $\tau:\tilde M^{1/r,\varphi}_{g,\gamma}
\to\tilde M^{1/r}_{g,\gamma} $ is equivariant.

\begin{lemma}
  \label{lem_fix_varphi}
  The fixed-point components of $\tilde M^{1/r,\varphi}_{g,\gamma}$ are
  \[
   F^{\varphi}_\star: = \tau^{-1}(F^{1/r}_{\star}), \quad \star = 0,\infty, J
  \]
  where $J$ runs over all subsets of $\{1,\dots,n\}$ such that $\{1\}\subsetneqq J$.
\end{lemma}
\begin{proof}
  Since $\tau$ is equivariant, it suffices to show that the $\mathbb C^*$-action
  on $F^{\varphi}_\star$ is trivial. For $J=0,\infty$, it follows immediately from the definition of $F^{1/r}_\star$.
  For $\star=J$, the underlying curve $C$ can be written as $C_J \cup E$ where $E$
  is the smooth rational subcurve containing $y_1$. The $\mathbb C^*$ acts trivially on $C_J$.
  For degree reasons, the $\varphi$-fields vanish on $E$. Hence the $\mathbb
  C^*$-action on $\tau^{-1}(F^{1/r}_{J})$ is trivial.
\end{proof}

By \cite{graber1999localization, kiem2013localizing},
the restriction of $\mathbb E^\vee_{\tilde
  M^\varphi}$ to $F^{\varphi}_\star$
decomposes as the direct sum of its fixed part and moving part.
The fixed part gives a perfect obstruction theory of $F^{\varphi}_\star$. The
cosection $\sigma$ restricts to a cosection on  $F^{\varphi}_\star$.
This gives a cosection localized virtual fundamental class
\[
  [F^{\varphi}_\star]_{\mathrm{loc}}^{\mathrm{vir}} \in A_*(F^{1/r}_\star).
\]
The moving part is the virtual normal bundle $\mathcal{N}^{\varphi}_\star$.
Let $\imath_\star:F^{1/r}_\star \to \tilde{M}^{1/r}_{g,\gamma}$ be
the inclusion.
We have the virtual localization formula
\begin{equation}
  \label{localization}
  [\tilde M^{1/r,\varphi}_{g,\gamma}]^{\mathrm{vir}}_{\mathrm{loc}} =
  \sum_{\star} (\imath_\star)_*
  \left(
    \frac{
      [F^{\varphi}_\star]_{\mathrm{loc}}^{\mathrm{vir}}
      }
      {c_{\mathrm{top}}^{\mathbb C^*}(\mathcal{N}^{\varphi}_\star|_{F^{1/r}_\star})}\right)
  \quad
  \text{ in  }
  A^{\mathbb C}_*(\tilde M^{1/r}_{g,\gamma})\otimes_{\mathbb Q[z]}\mathbb Q[z,z^{-1}],
\end{equation}
and
where $z\in A^1(B \mathbb C^*)$ is the first Chern class of $\mathbb C_1$, the
standard $\mathbb C$ with weight-$1$ action.

We have a stabilization map
\begin{equation}
  \label{st}
  \mathrm{st}:  \tilde M^{1/r}_{g,\gamma} \longrightarrow  \overline M^{1/r}_{g,\gamma}
\end{equation}
defined by forgetting the $(N,v_1,v_2)$ and then contracting the
unstable rational subcurves.
To construct this map, one can use the alternative description of of $r$-spin
curves as
balanced maps to Deligne--Mumford stacks in the proof of Theorem
\ref{proper_spin_master} (or in \cite{10.2307/1194469}). Then one can use
Corollary 9.1.3 of \cite{abramovich2002compactifying}.

The stabilization map $\mathrm{st}$ is $\mathbb C^*$-equivariant, where $\mathbb
C^*$  acts on $\overline M_{g,\gamma}^{1/r}$ trivially.
For any equivariant Chow cohomology class $\alpha\in A^*_{\mathbb
  C}(\tilde{M}^{1/r}_{g,\gamma})$,
we can cap both sides  of (\ref{localization}) with $\alpha$ and then
push it forward along the map $\mathrm{st}$. This gives us an equation
in
\[
  A^{\mathbb C^*}_*(\overline M^{1/r}_{g,\gamma}) \otimes_{\mathbb Q[z]}\mathbb
  Q[z,z^{-1}]= A_*(\overline M^{1/r}_{g,\gamma})
  \otimes_{\mathbb Q} \mathbb Q [z,z^{-1}].
\]
Since $ \mathrm{st}_*(\alpha\cap[\tilde
M^{1/r,\varphi}_{g,\gamma}]^{\mathrm{vir}}_{\mathrm{loc}})$ lies in $A^{\mathbb C^*}_*(\overline M^{1/r}_{g,\gamma}) =A_*(\overline
M^{1/r}_{g,\gamma}) \otimes_{\mathbb Q} \mathbb Q [z]$, we have
\begin{proposition}
  \label{prop_vanishing}
  For any $\alpha \in A^*_{\mathbb C}(\tilde{M}^{1/r}_{g,\gamma})$,
  the coefficients of negative degree powers of $z$ in
  \begin{equation}
    \label{relation_general}
    \mathrm{st}_*\bigg(   \sum_{\star} (\imath_\star)_*
      \bigg(
        \frac{
          (\imath_\star)^*\alpha \cap [F^{\varphi}_\star]_{\mathrm{loc}}^{\mathrm{vir}}
        }
        {c_{\mathrm{top}}^{\mathbb C^*}(\mathcal{N}^{\varphi}_\star|_{F^{1/r}_{\star}})}\bigg)
    \bigg) \in A_*(\overline M^{1/r}_{g,\gamma})
    \otimes_{\mathbb Q} \mathbb Q [z,z^{-1}]
  \end{equation}
  are zero.
\end{proposition}

\subsection{The contribution from each fixed-point component}
In this subsection we determine the contribution of each fixed-point component
$F^{\varphi}_\star$ to (\ref{relation_general}).  We will use the distinguished triangle (\ref{triangle})
to compute the fixed and moving parts of the restriction of $\mathbb
E^\vee_{\tilde M^\varphi}$ to each $F^{\varphi}_\star$.

We first consider $\star = 0,\infty$. Let
\[
  \gamma^\prime =
  \big(\frac{a_1}{r},\cdots,\frac{a_m}{r},\frac{b_1}{r}\big|\frac{b_2}{r},\cdots,\frac{b_n}{r}\big).
\]
Parallel to  Lemma \ref{normal_r_spin_0_infty}, we have isomorphisms
\[
  F^{\varphi}_0 \cong \overline M^{1/r,\varphi}_{g,\gamma^\prime} \quad
  \text{and}
  \quad F^{\varphi}_\infty \cong \overline{M}^{1/r,\varphi}_{g,\gamma}.
\]
The first isomorphism adds orbifold structure of index $r^\prime =
r/\gcd(r,k)$ along the $(m+1)$-th marking $y_1$. By abuse
of notation we will also use $\mathrm{st}$ to denote the stabilization map
$\mathrm{st}:\overline{M}^{1/r}_{g,\gamma^\prime} \to
\overline M^{1/r}_{g,\gamma}$. This map is the
restriction of the previous stabilization map (\ref{st}) to $F^{1/r}_0
\cong\overline{M}^{1/r}_{g,\gamma^\prime}$.

\begin{lemma}
  \label{normal_12}
  Under the isomorphisms above, we have
   \[
    \frac{
      [F^{\varphi}_0]_{\mathrm{loc}}^{\mathrm{vir}}
    }
    {c_{\mathrm{top}}^{\mathbb C^*}(\mathcal{N}^{\varphi}_0|_{F^{1/r}_0})}
    =
    \frac{
        [\overline{M}^{1/r,\varphi}_{g,\gamma^\prime}]^{\mathrm{vir}}_{\mathrm{loc}}}
      {z-\psi_{y_1}}
      \quad \text{and} \quad
    \frac{
      [F^{\varphi}_\infty]_{\mathrm{loc}}^{\mathrm{vir}}
    }
    {c_{\mathrm{top}}^{\mathbb C^*}(\mathcal{N}^{\varphi}_\infty|_{F^{1/r}_\infty})}
    =
    \frac{
      [\overline{M}^{1/r,\varphi}_{g,\gamma}]^{\mathrm{vir}}_{\mathrm{loc}}}
    {-z+\psi_{y_1}}
  \]
  where $\psi_{y_1}$ means the cotangent-line class at $y_1$ on the coarse curves.
\end{lemma}
\begin{proof}
  For $\star=0,\infty$, we restrict the distinguished triangle (\ref{triangle})
  to ${F^{\varphi}_\star}$.
  The complex $\mathbb E_{\tau}^\vee|_{F^{\varphi}_\star}$
  only has fixed part, since the action on the $r$-spin curves is trivial
  over $F^{\varphi}_\star$. It is identified with the relative obstruction theory of
  $\overline{M}^{1/r,\varphi}_{g,\gamma^\prime}$ or
  $\overline{M}^{1/r,\varphi}_{g,\gamma}$,
  compatible with the cosection in Section \ref{obstruction_theory_single}.
  The fixed and moving parts of $\tau^*T_{\tilde M^{1/r}}|_{F^{\varphi}_\star}$
  are the pullback of the tangent and normal bundle of $F^{1/r}_\star$, respectively.
  The normal bundle of $F^{1/r}_\star$ in $\tilde M^{1/r}$ is computed in Lemma
  \ref{normal_r_spin_0_infty}.
  Hence the lemma follows.
\end{proof}

We now come to $\star=J$ for $\{1\}\subsetneqq J\subset\{1,\cdots,n\}$. Recall
that
\[
  F^{1/r}_J \cong
  \overline{M}^{1/r}_{g,\gamma_J}
  \!\!\!
  \underset{\phantom{a}B\pmb\mu_{r/r^\prime}}{\times}
  \!\!\!
  \tilde M^{1/r}_{0,\gamma_E},
\]
where $r^\prime = r/\gcd(r,k)$.
Since the
$\varphi$-fields vanish on the rational tails for degree reasons, we have the
following fibred diagram
\[
  \xymatrix{
    F^{\varphi}_J \ar[r]^{\mathrm{pr}_1^\varphi} \ar[d]&  \overline{M}^{1/r,\varphi}_{g,\gamma_J} \ar[d]^\tau\\
    F^{1/r}_J \ar[r]^{\mathrm{pr}_1} &   \overline{M}^{1/r}_{g,\gamma_J}
  }
\]

Define
\[
  \ell_{\alpha,J} = \bigg \lfloor  \sum_{j\in J}\langle
  q_\alpha(b_j-1)\rangle\bigg \rfloor \quad \text{and} \quad  k_{\alpha,J}  =
  q_\alpha + \Big\langle \sum_{j\in J} q_\alpha(b_j-1) \Big\rangle,
\]
where $\langle x \rangle=x-\left \lfloor x \right \rfloor $ denotes the
fractional part of $x$.
Define
\[
  \mu_J(z) = \prod_{\alpha=1}^s \left[ k_{\alpha,J} \right]_{\ell_{\alpha,J}} z^{1-|J|+ \sum_{\alpha} \ell_{\alpha,J}},
\]
where $[x]_n =x(x+1)\cdots(x+n-1)$.
% Let $\mu^+_J(z)$ be the truncation of $\mu_J(z)$ consisting of all non-negative powers of $z$.
\begin{lemma}
  \label{normal_3}
  We have
  \[
      \frac{
      [F^{\varphi}_J]_{\mathrm{loc}}^{\mathrm{vir}}
    }
    {c_{\mathrm{top}}^{\mathbb C^*}(\mathcal{N}^{\varphi}_J|_{F^{1/r}_J})}
    =
    \mathrm{pr}_1^*\left((-1)^{\Sigma_{\alpha} \ell_{\alpha,J}}
      \frac{r^\prime \mu_J(-z)
      }
      {z-\psi_{x_J}}
      \cap
      [\overline{M}_{g,\gamma_J}^{1/r,\varphi}]^{\mathrm{vir}}_{\mathrm{loc}}
    \right),
  \]
  where $\psi_{x_J}$ is the cotangent-line class at the new heavy marking $x_J$ on the coarse curves.
\end{lemma}
\begin{proof}
The distinguished triangle (\ref{triangle}) restricts to a distinguished triangle
on $F_J^{\varphi}$. We will compute the fixed and moving parts of the first and
third terms.
The fixed parts will give us
$\mathrm{pr}_1^*([\overline{M}^{1/r,\varphi}_{g,\gamma_J}]_{\mathrm{loc}}^{\mathrm{vir}})
= [F^{\varphi}_J]^{\mathrm{vir}}_{\mathrm{loc}}$. To show this, by Proposition 7.5 of
\cite{behrend97_intrin_normal_cone}, it suffices to show that the fixed part of
$\mathbb E^{\vee}_{\tilde  M^\varphi}|_{F^{\varphi}_J}$ is the pulls back of the perfect obstruction
theory on  $\overline{M}^{1/r,\varphi}_{g,\gamma_J}$ defined in
Section \ref{obstruction_theory_single}, compatible with the cosections.
\footnote{Proposition 7.5 of
  \cite{behrend97_intrin_normal_cone} also works in our cosection localization setting. This is because
  $\mathrm{pr}_1$ is flat and the intrinsic normal cone pulls back to the
  intrinsic normal cone.}

We introduce some notation. Let $\mathcal  C=\mathcal  C_J\cup
\mathcal  E \to    F^{\varphi}_J$ be the pullback of the universal curve, where
$\mathcal  C_J$ is the pullback of the universal curve over
$\overline{M}^{1/r,\varphi}_{g,\gamma_J}$ and $\mathcal  E$ is the
pullback of the universal curve of $\tilde M^{1/r,\varphi}_{0,\gamma_E}$.  Let
$\mathcal  L$ be the restriction of the universal $r$-spin line bundle to $\mathcal
C$.
Let $x_i,y_j,D_\alpha$ denote the divisors on $\mathcal  C$ as in Definition
\ref{defn_master_varphi}. Recall that
\[
  D_{\alpha} = - \sum_{q_\alpha a_i \in \mathbb
    Z}x_i+\sum_{b_j\neq r}\left \lfloor
    q_\alpha (b_j-1) \right \rfloor y_j -\sum_{b_j= r}y_j.
\]

We first study the relative perfect obstruction theory
\[
  \mathbb E^\vee_{\tau}|_{F^{\varphi}_J} \cong   R\pi_*\left(
    \oplus_{\alpha=1}^s \mathcal
    L^{w_\alpha}(D_\alpha)\right).
\]
For each $\alpha$, consider the short exact sequence
\[
  0 \longrightarrow \mathcal  L^{w_\alpha}(D_\alpha)|_{\mathcal  C_J} (-x_J) \longrightarrow \mathcal  L^{w_\alpha}(D_\alpha) \longrightarrow \mathcal
  L^{w_\alpha}(D_\alpha)|_{\mathcal  E} \longrightarrow 0
\]
where the first and third terms are understood as their pushforward to $\mathcal  C$.
This induces the distinguished triangle
\[
  R\pi_*\big(\mathcal L^{w_\alpha}(D_\alpha)|_{\mathcal C_J} (-x_J)
  \big)\longrightarrow R\pi_*(\mathcal L^{w_\alpha}(D_\alpha)) \longrightarrow
  R\pi_*(\mathcal L^{w_\alpha}(D_\alpha)|_{\mathcal  E} ) \overset{+1}{\longrightarrow}.
\]

We will compute $R\pi_*(\mathcal L^{w_\alpha}(D_\alpha)|_{\mathcal  E} )$  and
show that it is in the moving part. Since $\mathcal
L^{w_\alpha}(D_\alpha)|_{\mathcal E}$ has negative fibre-wise degree, we have
\[
  R\pi_*\big(\mathcal L^{w_\alpha}(D_\alpha)|_{\mathcal E} \big) =
    R^1\pi_*\big(\mathcal L^{w_\alpha}(D_\alpha)|_{\mathcal E} \big)[-1].
  \]
  The datum $\mathcal  E$, $\mathcal  L|_{\mathcal  E}$ and so on are pulled back
  from the universal curve over $\tilde M^{1/r,\varphi}_{0,\gamma_E}\cong
B\pmb\mu_r$.
  Moreover we only care about Chow groups with rational coefficients, hence we can compute on
  a fixed smooth rational orbifold curve $E$, with $r$-spin bundle $L$.\footnote{The $\mathbb
    C^*$-action on $L$ is only well-defined up to $\pmb\mu_r$. Nevertheless we
    can compose the action with the $r$-th power map $\mathbb C^*\to \mathbb
    C^*$ so that it is well-defined.}
 Fixing $E$ and $L$ amounts to pulling back everything
  along the degree-$r$ cover $\mathrm {Spec\!~}\mathbb C \to   B\pmb\mu_{r}$.
  Let $x_\infty,y_1$ still denote the markings on $E$. Thus $E$ only has an orbifold point
  of index $r^\prime$ at $x_\infty$.  We choose the coordinate on the coarse
  moduli of $E$ such that  $y_1$ is at $0$
  and $x_\infty$ is at $\infty$. Thus the action on $E$ is given by
  (\ref{action_on_P1}).
  We define a divisor on $E$
  \[
    D_{\alpha,E} = \sum_{j\in J,b_j\neq r}\left \lfloor q_{\alpha}(b_j-1) \right
    \rfloor y_1-\sum_{j\in J,b_j =r}y_1.
  \]
  Up to the degree-$r$ cover introduced above, the line bundle $\mathcal  L|_{\mathcal  E}$ is the pullback
   of  $L$, and the divisor $D_\alpha|_{\mathcal  E}$ is the
   pullback of  $D_{\alpha,E}$.
  We have isomorphisms of equivariant line bundles
  \[
    L^{r} \cong \omega_{E}\big(x_\infty+\!\!\!\sum_{j\in J,b_j\neq r}\!\!(1-b_j)y_1+\!\!\!\sum_{j\in
      J,b_j=r}\!\!\! y_1\big), \quad \omega_E \cong \mathcal  O_E(-x_\infty -y_1).
  \]
  Hence we have
  \begin{align*}
    ((L^{w_\alpha}(D_{\alpha,E}))^{-1}\otimes& \omega_E)^r \cong \mathcal
                                                O_E\Big(-rx_\infty + r(k_{\alpha,J}+\ell_{\alpha,J} -1) y_1 \Big)\\
      \cong & \mathcal O_E\Big((r^\prime k_{\alpha,J} -1)x_\infty + (\ell_{\alpha,J}-1)y_1
                                                        \Big)^r \otimes \mathcal O_E(rk_{\alpha,J} y_1 - rr^\prime k_{\alpha,J} x_\infty
                                                        ).
  \end{align*}
  Note that $\mathcal O_E(rk_{\alpha,J} y_1 - rr^\prime k_{\alpha,J} x_\infty)$
  is the trivial bundle with the $\mathbb C^*$-action of weight $-rk_{\alpha,J}$.
  The equivariant Picard group of $E$ has no torsion. Hence we have
  \[
    \big(L^{w_\alpha}(D_{\alpha,E})\big)^{-1}\otimes \omega_E \cong \mathbb
    C_{-k_{\alpha,J}}\otimes_{\mathbb C} \mathcal  O_E\big(
   (r^\prime k_{\alpha,J} - 1)x_\infty
                                               + (\ell_{\alpha,J}-1)y_1
    \big).
  \]
  Since $0\leq  r^\prime k_{\alpha,J}-1<r^\prime$ and $x_\infty$
  is an orbifold point of index $r^\prime$, we can compute
  \[
    H^1(E,L^{w_\alpha}(D_E)) \cong \left(
      H^0(E,\big(L^{w_\alpha}(D_E)\big)^{-1}\otimes \omega_E)\right)^\vee
    \cong  \mathbb C_{k_{\alpha,J}}\otimes
    \bigg( \bigoplus_{i=0}^{\ell_{\alpha,J}-1} \mathbb C_{i} \bigg) .
  \]
  Hence it only has moving part and the equivariant top Chern class is
  \begin{equation}
    \label{moving_1}
    c_{\mathrm{top}}^{\mathbb C^*}\Big(R\pi_*\big(\mathcal
    L^{w_\alpha}(D_\alpha)|_{\mathcal E} \big)\Big)
    = \left(
      z^{\ell_{\alpha,J}}\left[k_{\alpha,J} \right]_{\ell_{\alpha,J}}
    \right)^{-1}.
  \end{equation}

  The term $R\pi_*\big(\mathcal L^{w_\alpha}(D_\alpha)|_{\mathcal C_J} (-x_J)
  \big)$ is in the fixed part, since the $\mathbb C^*$-action on $\mathcal C_J$ is
  trivial.
  We claim that it is the relative obstruction theory of
  $\overline{M}^{1/r,\varphi}_{g,\gamma_J}$ in Section \ref{obstruction_theory_single}.
  Indeed, if $\mathcal  L^{w_\alpha}$ has trivial monodromy at $x_J$ (i.e. $kq_\alpha\in
  \mathbb Z$), then by definition, the $\varphi$-fields
  of $\overline{M}^{1/r,\varphi}_{g,\gamma_J}$ are sections of $\mathcal
  L^{w_\alpha}(D_\alpha)|_{\mathcal  C_J} (-x_J)$;
  otherwise $\mathcal  L^{w_\alpha}$ has
  nontrivial monodromy at $x_J$, and thus we have
  \[
    R\pi_*\big(\mathcal L^{w_\alpha}(D_\alpha)|_{\mathcal C_J} (-x_J) \big)
    \cong R\pi_*\big(\mathcal L^{w_\alpha}(D_\alpha)|_{\mathcal C_J} \big).
  \]
  It is easy to see that the isomorphism is compatible with the cosections.

  Then we consider the $\tau^*\mathcal  T_{\tilde M}$ term of the distinguished
  triangle (\ref{triangle}). By Lemma~\ref{lem_fix_varphi}, the fixed part is $\tau^* \mathcal  T_{F^{1/r}_J}$. Moreover,
  since $\mathrm{pr}_1$ is \'etale,
  we have $\mathcal  T_{F^{1/r}_J}\cong
  \mathrm{pr}_1^*\mathcal  T_{\overline{M}^{1/r}_{g,\gamma_J}}$. This finishes the proof
  that $\mathrm{pr}_1^*([\overline{M}^{1/r,\varphi}_{g,\gamma_J}]_{\mathrm{loc}}^{\mathrm{vir}})
  = [F^{\varphi}_J]^{\mathrm{vir}}_{\mathrm{loc}}$.

  Similarly, the moving part of
  $\tau^*\mathcal  T_{\tilde M}|_{F^{\varphi}_J}$ is
  the pullback of the normal bundle of $F^{1/r}_J$ described in Proposition \ref{normal_spin}.
  Its equivariant top Chern class is
  \begin{equation}
    \label{moving_2}
    \frac{(-z)^{|J|-1}(z-\psi_{x_J})}{r^\prime}.
  \end{equation}
  Putting (\ref{moving_1}) and (\ref{moving_2}) together, we get the desired
  formula for the virtual normal bundle.
\end{proof}
\begin{remark}
  \label{rmk_twisted}
  In the case of twisted theories, the same lemma holds with  $\mu_J(z)$
  replaced by
  \[
    \mu^{\mathbb S}_J(z) = z^{1-|J|}\prod_{\alpha=1}^s
    \prod_{i=0}^{\ell_{\alpha,J}-1}((k_{\alpha,J}+i)z+w_\alpha \lambda_\alpha).
  \]
The rest of the paper works for twisted theories with this new definition .
\end{remark}
\subsection{The basic wall-crossing formula}
To describe the contribution of $F^{\varphi}_J$ in terms of weighted FJRW invariants, we define a map
\begin{equation}
  \label{beta_J}
  \beta_J:\overline{M}^{1/r}_{g,\gamma_J} \longrightarrow \overline{M}^{1/r}_{g,\gamma}.
\end{equation}
 The map replaces the last heavy marking $x_J$ by the set of light markings $\{y_j\}_{j\in J}$
 placed at the same point. More precisely, consider any family
\[\xi^\prime = (C^\prime,\pi^\prime,x^\prime_1,\cdots,x^\prime_m,x_J;(y^\prime_j)_{j\not\in
J},L^\prime,p^\prime) \in \overline{M}^{1/r}_{g,\gamma_J}(S),\] let
$\rho:C^\prime \to C$ be the partial coarse moduli only forgetting the orbifold
structure at $x_J$ and $\pi:C \to S$ be the induced projection. Set
$x_i=\rho(x_i^\prime)$ for $i=1,\cdots,m$, $y_j=\rho(y^\prime_j)$ for $j\not\in
J$ and $y_j = \rho(x_J)$ for $j\in J$. Recall that $r\ell+k=1+\sum_{j\in
J}(b_j-1)$, where $\ell\geq 0$, $1\leq k\leq r$ are integers. Let $L =
\rho_*(L^\prime)(-cy_1)$, where
\begin{equation}
  \label{c} c = \begin{cases} \ell - \big|\{j\in J:b_j =r\}\big| \cond{k\neq r} \\
    \ell - \big|\{j\in J:b_j =r\}\big|+1 \cond{k= r}\\
  \end{cases}.
\end{equation}
Then we have a natural isomorphism
\begin{align*}
  &f_*\Big((L^\prime)^{-r}\otimes \omega_{C^\prime}\big(x_J+\sum_{i=1}^m x^\prime_i +
  \sum_{\substack{j\not\in J\\b_j\neq r}}(1-b_j)y^\prime_j + \sum_{\substack{j\not\in
      J\\b_j=r}}y_j^\prime\big)\Big)\\ \cong~&
  L^{-r}\otimes \omega_{C}\big( \sum_{i=1}^m x_i +
  \sum_{b_j\neq r}(1-b_j)y_j + \sum_{b_j=r}y_j\big).
\end{align*}
Thus we get a possibly unstable family  \[\xi =
\big(C,\pi,x_1,\cdots,x_m;y_1,\cdots,y_n,L,f_*(p^\prime)\big).\]
The family $\beta_J(\xi^\prime)$ is obtained by stabilizing $\xi$.

From the concrete description of the maps, we have a commutative diagram
\[
  \xymatrix{
    F^{1/r}_{J}  \ar@{^{(}->}[r]\ar[d]_{\mathrm{pr}_1} & \tilde M^{1/r}_{g,\gamma}\ar[d]^{\mathrm{st}} \\
    \overline{M}^{1/r}_{g,\gamma_J} \ar[r]^{\beta_J} &\overline M^{1/r}_{g,\gamma}\\
  }
\]
Now recall the constant $\epsilon_{\gamma}$ defined in (\ref{defn_invariants}).
Let $d_1,\cdots,d_n$ be any non-negative integers. We use the short-hand
notation $\Sigma d_J$ for $\sum_{j\in J}d_j$.
 We use the projection formula to put Proposition \ref{prop_vanishing},
Lemma \ref{normal_12} and Lemma \ref{normal_3} together:
\begin{corollary}
  \label{cor_recursive}
  \begin{align*}
  & \prod_{j=1}^n \psi_{y_j}^{d_j}
  \cap
    \epsilon_{\gamma}[\overline{M}^{1/r,\varphi}_{g,\gamma}]^{\mathrm{vir}}_{\mathrm{loc}}
-
 \mathrm{st}_*\bigg(\prod_{j=1}^n \psi_{y_j}^{d_j}
    \cap
    \epsilon_{\gamma^\prime}[\overline{M}^{1/r,\varphi}_{g,\gamma^\prime}]^{\mathrm{vir}}_{\mathrm{loc}}
  \bigg)\\
  =& \sum_J\beta_{J*}\bigg(
    \Big[z^{\Sigma d_J}\mu_J(-z) \Big]_{+}\Big|_{z=\psi_J}\prod_{j\not\in J}^n
     \psi_{y_j}^{d_j}\cap \epsilon_{\gamma_J}
     [ \overline{M}_{g,\gamma_J}^{1/r,\varphi} ]^{\mathrm{vir}}_{\mathrm{loc}}
  \bigg),
  \end{align*}
where  the summation is over all $J\subset\{1,\cdots,n\}$ such that
$\{1\}\subsetneqq J$; the $\psi_{y_j}$ are the cotangent-line classes on the coarse curves.
\end{corollary}
\begin{proof}
  The equivariant $\psi$-classes on the master space $\tilde{M}^{1/r}_{g,\gamma}$ restrict
  to the non-equivariant $\psi$-classes on the fixed-point components, except for the
  following cases: for each $j\in J$, the class $\psi_{y_j}$ on
  $\tilde{M}^{1/r}_{g,\gamma}$
  restricts to $z$ on $F^{1/r}_J$. The corollary then follows from taking
    $\alpha=\prod_{j=1}^n \psi_{y_j}^{d_j}$,
  then taking the coefficient of $z^{-1}$-term in (\ref{relation_general}).
  Note that the map $\mathrm{pr}_1$ is finite flat of degree $1/r^\prime$. This cancels with
  the factor $r^\prime$ in Lemma \ref{normal_3}.
  Note that we have $\epsilon_\gamma=\epsilon_{\gamma^\prime} =
  (-1)^{\Sigma_\alpha \ell_{\alpha,J}}\epsilon_{\gamma_J}$.
\end{proof}
We now rewrite the lemma for later use.
We make the convention that when $J=\{1\}$, we have
$
\gamma_J =\gamma^{\prime},\beta_J=\mathrm{st},\mu_J(z) = 1,$ and $x_J=y_1,
$
then Corollary \ref{cor_recursive} can be rewritten as
\begin{equation}
  \label{cor_recursive_2}
  \prod_{j=1}^n \psi_{y_j}^{d_j} \cap
  \epsilon_{\gamma}[\overline{M}^{1/r,\varphi}_{g,\gamma} ]^{\mathrm{vir}}_{\mathrm{loc}}
  = \sum_J\beta_{J*}\bigg(
  \Big[z^{\Sigma d_J}\mu_J(-z) \Big]_{+}\Big|_{z=\psi_J}\prod_{j\not\in J}^n
    \psi_{y_j}^{d_j}\cap
    \epsilon_{\gamma_J}
    [\overline{M}_{g,\gamma_J}^{1/r,\varphi}]^{\mathrm{vir}}_{\mathrm{loc}}
  \bigg),
\end{equation}
where the summation is over all $1\in J\subset \{1,\cdots,n\}$.

\section{The wall-crossing formulas}
\label{wall_crossing_formulas}
\subsection{The Chow version}
In this subsection, we consider
\[
  \gamma^- =
  \big(\frac{a_1}{r},\cdots,\frac{a_m}{r}\big|\frac{b_1}{r},\cdots,\frac{b_h}{r},\frac{b_{h+1}}{r},\cdots,\frac{b_n}{r}\big)
\]
and
\[
  \gamma^+ =
  \big(\frac{a_1}{r},\cdots,\frac{a_m}{r},\frac{b_1}{r},\cdots,\frac{b_h}{r}\big |\frac{b_{h+1}}{r},\cdots,\frac{b_n}{r}\big).
\]
Assume that $h\neq 0$, $2g-2+m\geq 0$ and $a_i,b_j\in \{1,\cdots,r\}$ as before.
In both situations
we denote the markings by
\[
  (x_1,\cdots,x_m,y_1,\cdots,y_{h},y_{h+1},\cdots,y_n).
\]
As before we have a stabilization map
\[
  \mathrm{st}: \overline M_{g,\gamma^{+}}^{1/r} \longrightarrow \overline M_{g,\gamma^-}^{1/r}.
\]
Moreover we can generalize $\beta_J$ as follows. For any collection $\mathcal J\subset 2^{\{1,\cdots,n\}}$ of disjoint subsets of
$\{1,\cdots,n\}$, we denote $\bigcup_{J\in \mathcal  J}J$ by $\cup \mathcal J$,
and define
\[
  \gamma_{\mathcal J} = \big(
  \frac{a_1}{r},\cdots,\frac{a_m}{r}, (\frac{k_J}{r})_{J\in \mathcal  J} \big |
  (\frac{b_j}{r})_{j\not\in \cup \mathcal  J}
  \big),
\]
where $k_J$ is the integer such that
\[
  1\leq k_J\leq r \quad \text{and} \quad k_J-1\equiv\sum_{j\in J}(b_j-1)\mod r.
\]
We denote the $J$-th new heavy marking by $x_J$
and form a map
\[
  \beta_{\mathcal J}: \overline M^{1/r}_{g,\gamma_{\mathcal J}} \longrightarrow \overline M^{1/r}_{g,\gamma^-}.
\]
The definition of the map $\beta_{\mathcal  J}$ is similar to that of $\beta_J$ in (\ref{beta_J}).
For each $J\in \mathcal  J$,
we replace the heavy marking $x_J$ by the light markings $\{y_j\}_{j\in J}$.
Then we modify the line bundle as in (\ref{beta_J}), successively for each $J\in \mathcal  J$.
Finally we stabilize the family by contracting all unstable rational subcurves.

For various $\mathcal  J$, the maps $\beta_{\mathcal J}$  are compatible with each
other in the following sense. Suppose $\mathcal  J$ is the disjoint union of
$\mathcal  J_1$ and $\mathcal J_2$. If we view $\mathcal  J_1$ as a collection of
pairwise disjoint subsets of $\{1,\cdots,n\}\backslash \cup \mathcal  J_2$,
then we can compose $\beta_{\mathcal J_2}$ and $\beta_{\mathcal J_1}$. It is easy
to see that $\beta_{\mathcal J} =\beta_{\mathcal J_1}\circ \beta_{\mathcal J_2}$. Morover, when
$\mathcal  J=\{J\}$, we have $\beta_{\mathcal J} = \beta_{J}$. Hence $\beta_{\mathcal J}$ is
the composition of a sequence of $\beta_{J}$.

We now have the wall-crossing formula relating $\gamma^+$  and $\gamma^-$.
  For any integers $d_1,\cdots,d_n\geq 0$, recall that we write $\Sigma d_J = \sum_{j\in
    J} d_j$.
\begin{theorem}[Main Theorem]
  \label{thm_main}
  \begin{align*}
    \prod_{j=1}^{n}\psi_{y_j}^{d_j}\cap \epsilon_{\gamma^-}& [\overline
    M_{g,\gamma^-}^{1/r,\varphi} ]^{\mathrm{vir}}_{\mathrm{loc}} \\
     =&
       \sum_{\mathcal J}
    \beta_{\mathcal J*}\bigg(
    \prod_{j\not\in \cup \mathcal  J} \psi_{y_j}^{d_j}
        \prod_{J\in \mathcal J}
        \Big[ z^{\Sigma d_J}\mu_J(-z) \Big]_{+}\Big|_{z=\psi_J}
         \cap
    \epsilon_{\gamma_{\mathcal  J}}[\overline M^{1/r,\varphi}_{g,\gamma_{\mathcal
    J}}]^{\mathrm{vir}}_{\mathrm{loc}}\bigg),
  \end{align*}
    where $\mathcal J$ runs over all subsets of the power set
    $2^{\{1,\cdots,n\}}$ such that
    \begin{enumerate}
    \item the sets $J\in \mathcal  J$ are disjoint from each other,
    \item $\{1,\cdots,h\}\subset \cup \mathcal  J$,
    \item for each $J\in \mathcal J$, $J \cap \{1,\cdots,h\}\neq \varnothing$.
    \end{enumerate}
    \end{theorem}
\begin{proof}
  We prove this by induction on $h$. We
  will suppress the pushforward notation and everything will be pushed
  forward to $\overline M_{g,\gamma^-}^{1/r} $.

The case $h=1$ is Corollary \ref{cor_recursive}, reformulated as
(\ref{cor_recursive_2}). We now assume that $h>1$ and the theorem is
true for the wall-crossing between
$\gamma^-$ and
\[
  \gamma_0 =
\big(\frac{a_1}{r},\cdots,\frac{a_m}{r},\frac{b_1}{r},\cdots,\frac{b_{h-1}}{r}\big
|\frac{b_h}{r},\frac{b_{h+1}}{r},\cdots,\frac{b_n}{r}\big).
\]
This gives
  \begin{equation}
    \label{induction}
  \begin{aligned}
    \prod_{j=1}^{n}\psi_{y_j}^{d_j}\cap \epsilon_{\gamma^-}[\overline
    M_{g,\gamma^-}^{1/r,\varphi} ]^{\mathrm{vir}}_{\mathrm{loc}}
     =
       \sum_{\mathcal J_1}
        \prod_{j\not\in \cup \mathcal  J_1} \psi_{y_j}^{d_j}
        \prod_{J\in \mathcal J_1}
        \Big[ z^{\Sigma d_J}\mu_J(-z) \Big]_{+}\Big|_{z=\psi_J}
       \cap
    \epsilon_{\gamma_{\mathcal J_1}}[\overline
    M^{1/r,\varphi}_{g,\gamma_{\mathcal J_1}}]^{\mathrm{vir}}_{\mathrm{loc}},
  \end{aligned}
   \end{equation}
   where $\mathcal  J_1\subset 2^{\{1,\cdots,n\}}$ runs over all
   collections of pairwise disjoint subsets such that
  \[
    \{1,\cdots,h-1\}\subset \cup \mathcal  J_1 \quad \text{and} \quad J\cap
    \{1,\cdots,h-1\}\neq \varnothing,~\forall J\in \mathcal  J_1 .
  \]

  For each $\mathcal  J_1$ such that $h\not\in \cup \mathcal  J_1$, $\overline
  M^{1/r}_{g,\gamma_{\mathcal J_1}}$ has $y_h$ as a light marking. We can apply
  (\ref{cor_recursive_2}) to replace it with a heavy marking. This gives
  \begin{align*}
    & \prod_{j\not\in \cup \mathcal  J_1} \psi_{y_j}^{d_j}
    \prod_{J\in \mathcal J_1}\bigg(
        \Big[z^{\Sigma d_J}\mu_J(-z) \Big]_{+}\Big|_{z=\psi_J}
    \bigg) \cap
    \epsilon_{\gamma_{\mathcal  J_1}}[\overline M^{1/r,\varphi}_{g,\gamma_{\mathcal
          J_1}}]^{\mathrm{vir}}_{\mathrm{loc}} \\
    = &\sum_{J_2}
      \prod_{\substack{j\not\in \cup \mathcal
          J_1\\
          j\not\in J_2}} \psi_{y_j}^{d_j}
        \Big[  z^{\Sigma d_{J_2}}\mu_{J_2}(-z) \Big]_{+}\Big|_{z=\psi_{J_2}}
      \prod_{J\in \mathcal J_1}
        \Big[ z^{\Sigma d_J}\mu_J(-z) \Big]_{+}\Big|_{z=\psi_{J}}
       \cap
    \epsilon_{\gamma_{\mathcal J_2}}[\overline
    M^{1/r,\varphi}_{g,\gamma_{\mathcal J_2}}]^{\mathrm{vir}}_{\mathrm{loc}},
  \end{align*}
  where
 $\mathcal  J_2 = \mathcal  J_1 \cup\{J_2\}$ and $J_2$ runs over all subsets of $\{1,\cdots,n\}\backslash(\cup
  \mathcal  J_1)$ such that $h\in J_2$.
  Since $\mathcal  J_1$
  and $\mathcal  J_2$ exactly run over all $\mathcal  J$ in the statement of the
  theorem,  the proof is complete.
\end{proof}

\subsection{The invariants and potentials}
In this subsection we prove the wall-crossing formula for the generating
functions of weighted FJRW invariants.
Let $2g-2+m\geq 0$, $(2g-2+m,n)\neq (0,0)$ as before and consider
\[
  \gamma^{-} = \big(\frac{a_1}{r},\cdots,\frac{a_m}{r}\big|\frac{b_1}{r},\cdots,\frac{b_n}{r}\big)
  \quad \text{and} \quad
  \gamma^{+} = \big(\frac{a_1}{r},\cdots,\frac{a_m}{r},\frac{b_1}{r},\cdots,\frac{b_n}{r} \big|~\varnothing~\big).
\]
Theorem \ref{thm_main} implies
\begin{corollary}
  \label{cor_invariants}
\begin{align*}
  \big\langle \psi^{c_1} \phi_{a_1},\cdots, \psi^{c_m}\phi_{a_m}|&
    \psi^{d_1}\phi_{b_1},\cdots,\psi^{d_n}\phi_{b_n} \big\rangle^{0}_{g,m|n}  \\
  =  \sum_{h=1}^\infty \frac{1}{h!}\!\!\!
     \sum_{J_1,\cdots,J_h}\!\!\!\!
     \big\langle \psi^{c_1}
     \phi_{a_1},&\!\cdots\!, \psi^{c_m}\phi_{a_m},\\
          \Big[ z^{\Sigma d_{J_1}} &\mu_{J_1}(-z) \Big]_{+}\Big|_{z=\psi}\phi_{k_{J_1}}
     ,\!\cdots\!,
        \Big[  z^{\Sigma d_{J_h}}\mu_{J_h}(-z) \Big]_{+}\Big|_{z=\psi}
     \phi_{k_{J_h}}
     \big\rangle^{\infty}_{g,m+h},
\end{align*}
where the sequence $(J_1,\cdots,J_h)$ runs over all  length-$h$ partitions of $\{1,\cdots,n\}$.
\end{corollary}
\begin{proof}
  The $\psi$-classes at the have markings pullback along the maps $\beta_{\mathcal
    J}$. The corollary is proved by capping both sides of the equation in
  Theorem \ref{thm_main} by $\psi_{x_1}^{c_1}\cdots\psi_{x_m}^{c_m}$ and them
  taking the degrees of the classes on both sides. In Theorem \ref{thm_main}, the $\mathcal
  J$ is a set;  while here $(J_1,\cdots,J_h)$ is a sequence.  This accounts for
  the factor $1/h!$.
\end{proof}

We can package these formulas into generating functions. We  only consider
primary insertions at the light markings.
Recall that we have defined
\[
  \mathcal F^{0}_{g}(\mathbf u,\mathbf t) = \sum_{m,n\geq 0}
  \frac{1}{m!n!}\langle \mathbf u^m| \mathbf t^n
  \rangle_{g,m|n}^{0} \quad \text{and} \quad \mathcal F^{\infty}_{g}(\mathbf u)
  =\mathcal F^{0}_{g}(\mathbf u,0),
\]
where
\[
  \mathbf u = u_0 + u_1 \psi + u_2\psi^2 + \cdots \quad\! \text{and}\! \quad  \mathbf t = \sum_{j=1}^r t_j\phi_j,\quad\text{where
  } u_i=\sum_{j=1}^r u_{ij}\phi_j.
\]
The unstable terms are defined to be zero.

We have also defined
\[
  \mu(\mathbf t,z) = \sum_{n=1}^\infty \sum_{B_n}\frac{t_{b_1}\cdots t_{b_n}}{n!}
  \mu^+_{B_n}(z) \phi_{k_{B_n}}.
\]
Some explanation of the notation is in order. When we have fixed a
sequence of
positive integers $B=(b_1,\cdots,b_n),b_j\in \{1,\cdots,r\}$ and an index set $J\subset
\{1,\cdots,n\}$,
we define $k_J$ and $\ell_J$
to be the integers such that $\ell_J\geq
0,1\leq k_J\leq r$ and $r\ell_J+k_J = 1+\sum_{j\in J}(b_j-1)$. When we simply
write $k_B$ and $\ell_B$, we mean $k_J$,$\ell_J$ with $J=\{1,\cdots,n\}$. The
same rule applies to $\mu_J(z)$ and $\mu_B(z)$.

Now we are ready to prove the numerical wall-crossing formula  (Theorem \ref{thm_higher_genus_potential})
\begin{equation}
  \label{num_wall_crossing}
  \mathcal F^{0}_{g}(\mathbf u,\mathbf t) = \mathcal
  F^\infty_g(\mathbf u+\mu^+(\mathbf t,-\psi)).
\end{equation}
When $g=0$, this is only true modulo linear terms in the $\{u_{ij}\}$.
\begin{proof}[Proof of Theorem \ref{thm_higher_genus_potential}]

  We fix some $n\geq 0$ and $b_1,\cdots,b_n$ in $\{1,\cdots,r\}$.
  We compare the coefficients of $t_{b_1}\cdots t_{b_n}$ on both sides of
  (\ref{num_wall_crossing}).
  For each $m\geq 0$, we look at the $\{u_{ij}\}$-degree-$m$ part.
  Assume $m\geq 2$ when $g=0$ and $m\geq 1$ when $g=1$. The left hand side
  of (\ref{num_wall_crossing}) gives
  \[
    \frac{1}{\phantom{\!\!a}\big|\mathrm{Aut}(b_1,\cdots,b_n)\big|\phantom{\!\!a}} \langle
    \mathbf u^m
    | \phi_{b_1},\cdots,\phi_{b_n} \rangle_{g,m|n}^{0},
  \]
  where $\mathrm{Aut}(b_1,\cdots,b_n)$ is the group of permutations of
  $\{1,\cdots,n\}$ that fix $(b_1,\cdots,b_n)$.
  By Corollary \ref{cor_invariants}, this equals to
  \begin{equation}
    \begin{aligned}
    &\frac{1}{\phantom{\!\!a}\big|\mathrm{Aut}(b_1,\cdots,b_n)\big|\phantom{\!\!a}}
    \sum_{h=1}^\infty \frac{1}{h!}\sum_{J_1,\cdots,J_h}\big\langle \mathbf u^m,
    \mu^+_{J_1}(-\psi)\phi_{k_{J_1}},\cdots,\mu^+_{J_h}(-\psi)\phi_{k_{J_h}}
    \big\rangle_{g,m+h}^{\infty} \\
    =&
     \frac{1}{\phantom{\!\!a}\big|\mathrm{Aut}(b_1,\cdots,b_n)\big|\phantom{\!\!a}}
     \sum_{h=1}^\infty \frac{1}{h!}\sum_{J^\prime_1,\cdots,J^\prime_h}\Big\langle
       \mathbf u^m,
     \frac{\mu^+_{J^\prime_1}(-\psi)\phi_{k_{J^\prime_1}}}{|J^\prime_1|!},\cdots,\frac{\mu^+_{J^\prime_h}(-\psi)\phi_{k_{J^\prime_h}}}{|J^\prime_h|!}
    \Big\rangle_{g,m+h}^{\infty}
    \end{aligned}
  \end{equation}
  where $(J_1,\cdots,J_h)$ runs over all partitions of $\{1,\cdots,n\}$ of
  length $h$. The $(J^\prime_1,\cdots,J^\prime_h)$ also
  runs over all partitions of $\{1,\cdots,n\}$ of length $h$ but each $J^\prime_i$ is viewed as
  a sequence.

  The right hand side of (\ref{num_wall_crossing}) gives
  \[
    \sum_{h=1}^\infty \frac{1}{h!}\sum_{B_1,\cdots,B_h} \Big\langle
      \mathbf u^m,
      \frac{\mu^+_{B_1}(-\psi)\phi_{k_{B_1}}}{|B_1|!},\cdots,\frac{\mu^+_{B_h}(-\psi)\phi_{k_{B_h}}}{|B_h|!}
    \Big\rangle_{g,m+h}^\infty,
  \]
  where the $(B_1,\cdots,B_h)$  runs over all sequences of  sequences of numbers in $\{1,\cdots,r\}$ such that
  the juxtaposition of $B_1,\cdots,B_h$ is equal to $(b_1,\cdots,b_n)$ up to permutation.

  The proof is now complete by noticing that given $B_1,\cdots,B_h$, there are
  exactly $\big|\mathrm{Aut}(b_1,\cdots,b_n)\big|$ many choices of
  $J_1^\prime,\cdots,J_h^\prime$ such that the sequence $B_i$ is equal to the
  sequence $(b_j)_{j\in J^\prime_i} $, for
  all $i=1,\cdots,h$.
\end{proof}
\begin{remark}
  \label{rmk_narrow}
  Since the invariants with broad heavy insertions vanish (Lemma \ref{lem_Ramond_vanishing}), the generating
  function  $\mathcal F^{0}_{g}(\mathbf u,\mathbf t)$ is independent of $u_{ij}$
  unless
  \[
    jq_\alpha \not\in \mathbb Z, \text{  for all }\alpha = 1,\cdots,s.
  \]
  In $\mathcal F^\infty_g(\mathbf u+\mu^+(\mathbf t,-\psi))$, the invariants
  involving broad $\phi_{k_{B_n}}$ will be zero. Hence we can redefine
  \[
    k_{\alpha,B_n} := \big\langle  q_\alpha k_{B_n} \big\rangle,
  \]
  so that in $\mu(\mathbf t,z)$, the coefficients of narrow $\phi_{k_{B_n}}$ are
  unchanged and the coefficients of broad $\phi_{k_{B_n}}$ become zero. Thus the
  wall-crossing formula remains unchanged.
\end{remark}
\section{The genus-$0$ wall-crossing formulas}
\label{genus_0_case}
\subsection{The moduli space and virtual localization}
In this section we consider the case $g=0$ with one heavy marking and $n$ light
markings. We assume that  $n\geq 2$ and consider
\[
  \gamma = \big(
  \frac{a}{r}\big|\frac{b_1}{r},\cdots,\frac{b_n}{r}
  \big).
\]
In this case, the moduli space $\overline{M}_{0,\gamma}^{1/r}$ does not exist.
However the master space $\tilde M_{0,\gamma}^{1/r}$ exists if and only if
\begin{equation}
  \label{selection_genus_0}
   -2 + (1-a) + \sum_{b_j\neq r} (1-b_j) +\sum_{b_j=r}1 \equiv 0 \mod r.
\end{equation}
Theorem \ref{proper_spin_master} applies to this case. The perfect obstruction theory and
localization also work.  Compared to the list of fixed-point components in
Section \ref{C_star_spin}, we have a slightly different list here:
\begin{enumerate}
\item
  $F^{1/r}_0 = \{\xi:v_1=0\}$ remains the same;
\item There is no analogue of $F^{1/r}_\infty$. Indeed $v_2$ never vanishes.
\item
  For each $\{1\}\subsetneqq J\subset\{1,\cdots,n\}$, we have
\begin{itemize}
\item if $J\neq \{1,\cdots,n\}$,  $F_{J}^{1/r}$ remains the same;
\item if $J=\{1,\cdots,n\}$, $F^{1/r}_J$ consists of $(C,\pi,x_1;\mathbf
  y,N,L,v_1,v_2,p)$, where $\pi:C\to S$ is smooth; $y_j=y_1$ for all
  $j=1,\cdots,n$; $v_1,v_2$ are both non-vanishing.
\end{itemize}
\end{enumerate}

The description of  $F^{1/r}_0$ (Lemma \ref{normal_r_spin_0_infty}) and
$F^{1/r}_J,J\neq\{1,\cdots,n\}$ (Lemma \ref{stack_F1/r_J}) still
holds. For $J=\{1,\cdots,n\}$, it is easy to see that $F^{1/r}_J\cong B\pmb \mu_r$.

The $\varphi$-fields identically vanish
in the genus-$0$ case for degree reasons. The proof
is the same as that of Lemma 1.5 of \cite{ross2014wall}. Hence we have $\tilde
M^{1/r,\varphi}_{0,\gamma} = \tilde M^{1/r}_{0,\gamma}$.
We still have the localization formula (\ref{localization}). Let $\mathcal{N}^\varphi_\star$
be the virtual normal bundle of
$F^{\varphi}_\star$. For $\star = 0$ or $\star=J\neq\{1,\cdots,n\}$,
the $\mathcal{N}^\varphi_\star$ are the same as those in Lemma
\ref{normal_12} and Lemma \ref{normal_3}.
\begin{lemma}
  For $J=\{1,\cdots,n\}$, we have
  \[
    \frac{1}{c_{\mathrm{top}}^{\mathbb C^*}(\mathcal{N}^\varphi_J)} =
  (-1)^{\Sigma_{\alpha}\ell_{\alpha,J}}\mu_B(-z).
\]
\begin{proof}
  The computation of the normal bundle is almost the same as that in Lemma~\ref{normal_3}. The
  relative obstruction theory is the same. The moving part of $\tau^*T_{\tilde
    M_{0,\gamma}}$ is $T_{y_1}^{\oplus (n-1)}$. Hence \[\frac{1}{c_{\mathrm{top}}^{\mathbb
      C^*}(\mathcal{N}^\varphi_J)}=
  (-1)^{\Sigma_{\alpha}\ell_{\alpha,J}}(-z)^{1-n+\Sigma_{\alpha}\ell_{\alpha,J}}
  \prod_{\alpha =1}^s \left[k_{\alpha,J}\right ]_{\ell_{\alpha,J}}=(-1)^{\Sigma_{\alpha}\ell_{\alpha,J}}\mu_B(-z).\]
\end{proof}
\end{lemma}
\subsection{The wall-crossing formulas}
We do not have the stabilization maps used in Proposition \ref{prop_vanishing}, due to the absence of $\overline
M_{0,\gamma}^{1/r}$. However, we can pushforward everything to a
point instead. For any non-negative integer $c$ this gives
\begin{equation}
  \label{genus_0_basic}
  \sum_J \Big\langle \psi^c\phi_a,\mu_J^+(-\psi) |(\phi_{b_j})_{j\not\in J} \Big\rangle^{0}_{0,2|(n-|J|)}
  =
    \big\{\mu_B(z)\big\}_{z^{-c-1}},
\end{equation}
where the summation is over all $1\in J\subsetneqq\{1,\cdots,n\}$, and
$\{*\}_{z^{-c-1}}$ means the coefficient of $z^{-c-1}$ in $*$.

For $a\in\{1,\cdots,r\}$, let $\phi^a=\phi_{a^\prime}$ where $a+a^\prime\equiv
0\mod r$.
Let $\mu^-_B(z)$ and $\mu^-(\mathbf t,z)$ be the truncation of $\mu_B(z)$ and
$\mu(\mathbf t,z)$ consisting of all the  negative powers of $z$.

\begin{corollary}
  \label{coro_genus_0_main}
  \begin{align*}
    \sum_{n\geq 2}\sum_{a=1}^r \frac{\phi^a}{n!} \big\langle \frac{\phi_a}{z-\psi},
    \big(\mu^+(\mathbf t,-\psi)\big)^n \big\rangle_{0,1+n}^\infty = \mu^-(\mathbf t,z).
  \end{align*}
\end{corollary}
\begin{proof}

We multiply both sides of (\ref{genus_0_basic}) by $z^{-c-1}$ and sum over  all
$c\geq 0$. This gives
\begin{equation}
  \label{genus_0_basis2}
  \sum_J \Big\langle \frac{\phi_a}{z-\psi},\mu_J^+(-\psi) |(\phi_{b_j})_{j\not\in J} \Big\rangle^{0}_{0,2|(n-|J|)}
  = \mu^-_B(z),
\end{equation}
where the summation is over all $1\in J\subsetneqq\{1,\cdots,n\}$.

For this proof only, we redefine the unstable terms
\[
  \langle \cdots \rangle^\infty_{0,2} :=0 \quad \text{and} \quad \Big\langle \frac{\phi_a}{z-\psi}\Big|\phi_{b_1},\cdots,\phi_{b_n} \Big\rangle^{0}_{0,1|n}
  := \mu^-_B(z).
\]
Thus (\ref{genus_0_basis2}) becomes
\[
  \sum_J \Big\langle \frac{\phi_a}{z-\psi},\mu_J^+(-\psi) |(\phi_{b_j})_{j\not\in J} \Big\rangle^{0}_{0,2|(n-|J|)}
  =
  \Big\langle \frac{\phi_a}{z-\psi}\Big|\phi_{b_1},\cdots,\phi_{b_n} \Big\rangle^{0}_{0,1|n},
\]
where the summation is over all $1\in J\subset\{1,\cdots,n\}$.
This is parallel to (\ref{cor_recursive_2}) and the combinatorics is exactly the
same as before. Hence we conclude that
\[
    \sum_{n\geq 2}\frac{1}{n!} \big\langle \frac{\phi_a}{z-\psi},
    \big(\mu^+(\mathbf t,-\psi)\big)^n \big\rangle_{0,1+n}^\infty = \sum_{n\geq
      2}\sum_{B_n} \frac{t_{b_1}\cdots t_{b_n}}{n!}\mu_{B_n}^{-}(z),
\]
where  $B_n=(b_1,\cdots,b_n)$ runs over all $n$-tuples
such that $b_j\in \{1,\cdots,r\}$ and $k_{B_n} +a \equiv 0 \mod r$. The latter
condition follows from (\ref{selection_genus_0}).

Hence we have $\phi^a = \phi_{k_{B_n}}$  for each $B_n$ and thus
\[
    \sum_{n\geq 2}\sum_{a=1}^r\frac{\phi^a}{n!} \big\langle \frac{\phi_a}{z-\psi},
    \big(\mu^+(\mathbf t,-\psi)\big)^n \big\rangle_{0,1+n}^\infty = \sum_{n\geq
      2}\sum_{B_n} \frac{t_{b_1}\cdots t_{b_n}}{n!}\mu_{B_n}^{-}(z)\phi_{k_{B_n}}.
\]
where  $B_n=(b_1,\cdots,b_n)$ runs over all $n$-tuples
of $b_j\in \{1,\cdots,r\}$.
Notice that $\mu_{B_1}(z)=1$ and thus $\mu^-_{B_1}(z)=0$. Hence we can extend
the summation on the right hand side to be over  all $n\geq 1$ and the Corollary follows.
\end{proof}

\subsection{The big $\mathcal  J$-function and Givental's formalism}
Now we define the big $\mathcal J$-function
\begin{equation}
  \label{defn_J}
  \mathcal  J^{0}(\mathbf u,\mathbf t,z) =
  \mathbf u(-z)+
  z\phi_1 + \mu(\mathbf t,z) + \sum_{m\geq 1}\sum_{n\geq 0}\sum_{a=1}^r \frac{1}{m!n!}\phi^{a}
  \big\langle
  \frac{\phi_a}{z-\psi},\mathbf u^m|\!~\mathbf
  t^n
  \big\rangle_{0,(1+m)|n}^{0},
\end{equation}
where the unstable term $(m,n)=(1,0)$ is defined to be zero.
We have
\begin{corollary}
  \[
    \mathcal  J^{0}(\mathbf u,\mathbf t,z)  =
   \mathbf u(-z)+\mu^+(\mathbf t,z)  +  z\phi_1 +
  \sum_{n\geq 0} \sum_{a=1}^r \frac{\phi^{a}}{n!}
  \big\langle
  \frac{\phi_a}{z-\psi},\big(\mathbf u + \mu^+(\mathbf t,-\psi)\big)^n
  \big\rangle_{0,1+n}^\infty,
\]
where the unstable terms $n=0,1$ are defined to be zero.
\end{corollary}
\begin{proof}
  We first write $\mu(\mathbf t,z)=\mu^+(\mathbf t,z)+\mu^-(\mathbf t,z)$ in
  (\ref{defn_J}).
  Putting Corollary \ref{coro_genus_0_main} and Theorem \ref{thm_higher_genus_potential} together, we get
  \begin{align*}
    &\mu^-(\mathbf t,z) +  \sum_{m\geq 1}\sum_{n\geq 0}\sum_{a=1}^r \frac{\phi^{a}}{m!n!}
    \big\langle
    \frac{\phi_a}{z-\psi},\mathbf u^m|\!~\mathbf t^n
    \big\rangle_{0,(1+m)|n}^{0}  \\
    =&\sum_{m\geq 0}\sum_{n\geq 0} \sum_{a=1}^r \frac{\phi^{a}}{m!n!}
  \big\langle
  \frac{\phi_a}{z-\psi},\mathbf u^m,\mu^+(\mathbf t,-\psi)^n
       \big\rangle_{0,1+m+n}^\infty \\
    =& \sum_{n\geq 0} \sum_{a=1}^r \frac{\phi^{a}}{n!}
  \big\langle
  \frac{\phi_a}{z-\psi},\big(\mathbf u+\mu^+(\mathbf
  t,-\psi)\big)^n\big\rangle_{0,1+m+n}^\infty
  \end{align*}
  This completes the proof.
\end{proof}

We can restrict our formula to the narrow state space since $\phi_a$ is narrow
if and only if $\phi^a$ is narrow. This means that we set $u_{ij}=t_j=0$ if
$jq_\alpha\in \mathbb Z$ for some $\alpha=1,\cdots,s$. The big $\mathbb
I$-function defined in \cite{ross2014wall} is our $z\phi_1+\mu^+(\mathbf t,z)$
(cf. Remark \ref{rmk_narrow}).
If we set $\widehat{\mathbf u}(z) = \mathbf u(z) + \mu(\mathbf t,-z)$, then $\mathcal  J^{0}(\mathbf u,\mathbf t,-z)$ satisfies Definition 1.10 in
\cite{ross2014wall}  and thus lies on the Lagrangian cone, matching \cite[Theorem 1]{ross2014wall}.
\bibliography{/home/yang/Dropbox/Research/references.bib}

\begin{thebibliography}{10}

\bibitem{abramovich2008lectures}
D.~Abramovich.
\newblock Lectures on {Gromov--Witten} invariants of orbifolds.
\newblock In {\em Enumerative invariants in algebraic geometry and string
  theory}, pages 1--48. Springer, 2008.

\bibitem{abramovich2008gromov}
D.~Abramovich, T.~Graber, and A.~Vistoli.
\newblock {Gromov--Witten} theory of {Deligne--Mumford} stacks.
\newblock {\em American Journal of Mathematics}, pages 1337--1398, 2008.

\bibitem{10.2307/1194469}
D.~Abramovich and T.~J. Jarvis.
\newblock Moduli of twisted spin curves.
\newblock {\em Proceedings of the American Mathematical Society},
  131(3):685--699, 2003.

\bibitem{abramovich2002compactifying}
D.~Abramovich and A.~Vistoli.
\newblock Compactifying the space of stable maps.
\newblock {\em Journal of the American Mathematical Society}, 15(1):27--75,
  2002.

\bibitem{behrend1997gromov}
K.~Behrend.
\newblock {Gromov--Witten} invariants in algebraic geometry.
\newblock {\em Inventiones Mathematicae}, 127(3):601--617, 1997.

\bibitem{behrend97_intrin_normal_cone}
K.~Behrend and B.~Fantechi.
\newblock The intrinsic normal cone.
\newblock {\em Inventiones Mathematicae}, 128(1):45--88, 1997.

\bibitem{chang2015witten}
H.-L. Chang, J.~Li, and W.-P. Li.
\newblock {Witten's} top {Chern} class via cosection localization.
\newblock {\em Inventiones Mathematicae}, 200(3):1015--1063, 2015.

\bibitem{chang2015mixed}
H.-L. Chang, J.~Li, W.-P. Li, and C.-C.~M. Liu.
\newblock Mixed-spin-p fields of {Fermat} quintic polynomials.
\newblock {\em arXiv preprint arXiv:1505.07532}, 2015.

\bibitem{chiodo2008stable}
A.~Chiodo.
\newblock Stable twisted curves and their $r$-spin structures.
\newblock In {\em Annales de l'institut Fourier}, volume~58, pages 1635--1689,
  2008.

\bibitem{Chiodo:2009hbu}
A.~Chiodo and D.~Zvonkine.
\newblock Twisted r-spin potential and {Givental's} quantization.
\newblock {\em Advances in Theoretical and Mathematical Physics}.

\bibitem{ciocan2010moduli}
I.~Ciocan-Fontanine and B.~Kim.
\newblock Moduli stacks of stable toric quasimaps.
\newblock {\em Advances in Mathematics}, 225(6):3022--3051, 2010.

\bibitem{ciocan2016quasimap}
I.~Ciocan-Fontanine and B.~Kim.
\newblock Quasimap wall-crossings and mirror symmetry.
\newblock {\em arXiv preprint arXiv:1611.05023}, 2016.

\bibitem{ciocan2014stable}
I.~Ciocan-Fontanine, B.~Kim, and D.~Maulik.
\newblock Stable quasimaps to {GIT} quotients.
\newblock {\em Journal of Geometry and Physics}, 75:17--47, 2014.

\bibitem{clader2017higherLG}
E.~Clader, F.~Janda, and Y.~Ruan.
\newblock Higher-genus wall-crossing in the gauged linear sigma model.
\newblock {\em arXiv preprint arXiv:1706.05038}, 2017.

\bibitem{fan2013witten}
H.~Fan, T.~Jarvis, and Y.~Ruan.
\newblock The {Witten} equation, mirror symmetry, and quantum singularity
  theory.
\newblock {\em Annals of Mathematics}, 178(1):1--106, 2013.

\bibitem{fan2015mathematical}
H.~Fan, T.~Jarvis, and Y.~Ruan.
\newblock A mathematical theory of the gauged linear sigma model.
\newblock {\em arXiv preprint arXiv:1506.02109}, 2015.

\bibitem{graber1999localization}
T.~Graber and R.~Pandharipande.
\newblock Localization of virtual classes.
\newblock {\em Inventiones mathematicae}, 135(2):487--518, 1999.

\bibitem{guo2016genus}
S.~Guo and D.~Ross.
\newblock Genus-one mirror symmetry in the {Landau--Ginzburg} model.
\newblock {\em arXiv preprint arXiv:1611.08876}, 2016.

\bibitem{guo2017genus}
S.~Guo and D.~Ross.
\newblock The genus-one global mirror theorem for the quintic threefold.
\newblock {\em arXiv preprint arXiv:1703.06955}, 2017.

\bibitem{hassett2003moduli}
B.~Hassett.
\newblock Moduli spaces of weighted pointed stable curves.
\newblock {\em Advances in Mathematics}, 173(2):316--352, 2003.

\bibitem{kiem2013localizing}
Y.-H. Kiem and J.~Li.
\newblock Localizing virtual cycles by cosections.
\newblock {\em Journal of the American Mathematical Society}, 26(4):1025--1050,
  2013.

\bibitem{knutson1971algebraic}
D.~Knutson and D.~Knutson.
\newblock {\em Algebraic spaces}.
\newblock Springer, 1971.

\bibitem{li1998virtual}
J.~Li and G.~Tian.
\newblock Virtual moduli cycles and {Gromov--Witten} invariants of algebraic
  varieties.
\newblock {\em Journal of the American Mathematical Society}, 11(1):119--174,
  1998.

\bibitem{marian2011moduli}
A.~Marian, D.~Oprea, and R.~Pandharipande.
\newblock The moduli space of stable quotients.
\newblock {\em Geometry and Topology}, 15(3):1651--1706, 2011.

\bibitem{mustacta2007intermediate}
A.~Musta{\c{t}}ǎ and M.~A. Musta{\c{t}}ǎ.
\newblock Intermediate moduli spaces of stable maps.
\newblock {\em Inventiones mathematicae}, 167(1):47--90, 2007.

\bibitem{olsson2016algebraic}
M.~Olsson.
\newblock {\em Algebraic Spaces and Stacks}.
\newblock Colloquium Publications. American Mathematical Society, 2016.

\bibitem{olsson2007}
M.~C. Olsson.
\newblock {(Log)} twisted curves.
\newblock {\em Compositio Mathematica}, 143(2):476--494, 03 2007.

\bibitem{polishchuk2004witten}
A.~Polishchuk.
\newblock {Witten}'s top {Chern} class on the moduli space of higher spin
  curves.
\newblock In {\em Frobenius manifolds}, pages 253--264. Springer, 2004.

\bibitem{ross2014wall}
D.~Ross and Y.~Ruan.
\newblock Wall-crossing in genus zero {Landau--Ginzburg} theory.
\newblock {\em Journal f{\"u}r die reine und angewandte Mathematik (Crelles
  Journal)}.

\bibitem{stacks-project}
T.~{Stacks Project Authors}.
\newblock {\itshape Stacks Project}.
\newblock \url{http://stacks.math.columbia.edu}, 2017.

\bibitem{thaddeus1996geometric}
M.~Thaddeus.
\newblock Geometric invariant theory and flips.
\newblock {\em Journal of the American Mathematical Society}, 9(3):691--723,
  1996.

\bibitem{toda2011moduli}
Y.~Toda.
\newblock Moduli spaces of stable quotients and wall-crossing phenomena.
\newblock {\em Compositio Mathematica}, 147(05):1479--1518, 2011.

\bibitem{witten1993phases}
E.~Witten.
\newblock Phases of n=2 theories in two dimensions.
\newblock {\em Nuclear Physics B}, 403(1-2):159--222, 1993.

\end{thebibliography}
\bibliographystyle{abbrv}
\end{document}